\documentclass[11pt, a4paper, oneside]{amsart}

\usepackage{slashed}
\usepackage{amsthm}
\usepackage{amsmath,amscd}
\usepackage{hyperref}
\usepackage{amssymb}

\theoremstyle{plain}% default 
\newtheorem{thm}{Theorem}[section] 
\newtheorem{lem}[thm]{Lemma} 
\newtheorem{prop}[thm]{Proposition} 
\newtheorem{cor}[thm]{Corollary} 
\theoremstyle{definition} 
\newtheorem{defn}[thm]{Definition}

\newtheorem{exmp}[thm]{Example}
\newtheorem{cond}[thm]{Condition}
\newtheorem*{ack}{Acknowledgements}
\theoremstyle{remark}

\DeclareMathOperator{\Hol}{Hol}

\DeclareMathOperator{\GL}{GL}
\DeclareMathOperator{\SL}{SL}

\DeclareMathOperator{\re}{Re}

\DeclareMathOperator{\Proj}{Proj}
\DeclareMathOperator{\Aut}{Aut}

\DeclareMathOperator{\Hom}{Hom}
\DeclareMathOperator{\linspan}{span}

\DeclareMathOperator{\Ad}{Ad}

\newcommand{\bbR}{\mathbb{R}}
\newcommand{\bbC}{\mathbb{C}}

\newcommand{\bbZ}{\mathbb{Z}}
\newcommand{\bbP}{\mathbb{P}}

\newcommand{\bbQ}{\mathbb{Q}}
\newcommand{\imi}{i}
\newcommand{\rd}{\mathrm{d}}
\begin{document}
\title{New Examples of Compact Manifolds with Holonomy $\mathrm{Spin}(7)$}
\author{Robert Clancy}
\email{clancy@maths.ox.ac.uk}
%\date{}                                           % Activate to display a given date or no date

\begin{abstract}
    We find new examples of compact $\mathrm{Spin}(7)$-manifolds using a construction of Joyce \cite{Joyce:1999fk,joyce2000compact}. 
    The essential ingredient in Joyce's construction is a Calabi--Yau 4-orbifold with particular singularities admitting an antiholomorphic involution, which fixes the singularities. 
    We search the class of well-formed quasismooth hypersurfaces in weighted projective spaces for suitable Calabi--Yau 4-orbifolds.
    We find that different hypersurfaces within the same family of Calabi--Yau 4-orbifolds may result in different $\mathrm{Spin}(7)$-manifolds.  
\end{abstract}

\maketitle
\section{Introduction}

The holonomy group of a connected Riemannian manifold is the group of parallel transport maps around piecewise smooth loops based at a point.
Berger \cite{Berger:1955fk}  classified the possible holonomy groups of irreducible, nonsymmetric Riemannian metrics on simply-connected manifolds in 1955. 

\begin{thm}[Berger]
Suppose $M$ is a simply-connected manifold and $g$ is a Riemannian metric on $M$, which is irreducible and nonsymmetric. Then one of the following cases holds:
\begin{itemize}
\item[(i)] $\Hol(g)=\mathrm{SO}(n)$,
\item[(ii)] $n=2m$ with $m\geq 2$, and $\Hol(g)=\mathrm{U}(m)$ in $\mathrm{SO}(2m)$,
\item[(iii)] $n=2m$ with $m\geq 2$, and $\Hol(g)=\mathrm{SU}(m)$ in $\mathrm{SO}(2m)$,
\item[(iv)] $n=4m$ with $m\geq2$, and $\Hol(g)=\mathrm{Sp}(m)$ in $\mathrm{SO}(4m)$,
\item[(v)] $n=4m$ with $m\geq 2$, and $\Hol(g)=\mathrm{Sp}(m)\mathrm{Sp}(1)$ in $\mathrm{SO}(4m)$,
\item[(vi)] $n=7$ and $\Hol(g)= G_2$ in $\mathrm{SO}(7)$, or
\item[(vii)] $n=8$ and $\Hol(g)=\mathrm{Spin}(7)$ in $\mathrm{SO}(8)$.
\end{itemize}
\end{thm}

The question of whether there existed manifolds with holonomy group $G_2$ or $\mathrm{Spin}(7)$ would not be resolved for more than 30 years.
Bryant \cite{Bryant:1987uq} in 1987 used the theory of exterior differential systems to show the existence of many metrics with holonomy $G_2$ and $\mathrm{Spin}(7)$ on small balls in $\bbR^7$ and $\bbR^8$, respectively. Then Bryant and Salamon 
\cite{Bryant:1989kx} constructed examples of complete metrics with holonomy $G_2$ and $\mathrm{Spin}(7)$ on non-compact manifolds, which were vector bundles over manifolds of dimensions 3 and 4. In 1994--5 Joyce  \cite{Joyce:1996gt,Joyce:1996st} constructed examples of compact manifolds with holonomy $G_2$ and $\mathrm{Spin}(7)$ by resolving quotients of tori by finite groups.

Joyce \cite{Joyce:1999fk} gives a second construction of manifolds with holonomy $\mathrm{Spin}(7)$ whose basic ingredient is a Calabi--Yau 4-orbifold with an antiholomorphic involution. The exact conditions on the Calabi--Yau 4-orbifold are stated in Condition \ref{cond:Y}.
%Manifolds with holonomy $\mathrm{Spin}(7)$ are interesting to string theorists because they provide spaces on which one can compactify M-theory??
In this thesis we will find all examples of suitable Calabi--Yau 4-orbifolds arising as well-formed quasismooth hypersurfaces in weighted projective spaces. We will then find the Betti numbers of the $\mathrm{Spin}(7)$-manifolds, which result from the construction given in \cite{Joyce:1999fk}. 

\begin{ack}
    I would like to thank my supervisor Dominic Joyce for encouragement, guidance and help. 
\end{ack}

\section {Review of $\mathrm{Spin}(7)$ geometry}
The material of this section is entirely from \cite{joyce2000compact}. We will recall the basic definitions and properties of Riemannian holonomy groups and then discuss the group $\mathrm{Spin}(7)$.
In this section $M$ will denote a connected manifold.

\begin{defn}
Let $E$ be a vector bundle over $M$, and $\nabla^E$ a connection on $E$. Let $p\in M$ be a point. We say $\gamma$ is a \emph{loop based at $p$} if $\gamma:[0,1]\rightarrow M$ is a piecewise-smooth curve with $\gamma(0)=\gamma(1)=p$. If $\gamma$ is a loop based at $p$, then the parallel transport map $P_\gamma:E_p\rightarrow E_p$ is an invertible linear map. Define the \emph{holonomy group $\Hol_p(\nabla^E)$ of $\nabla^E$ based at $p$} to be
\[
\Hol_p(\nabla^E)=\{P_\gamma:\text{ $\gamma$ is a loop based at $p$}\}\subset \GL(E_p).
\]
\end{defn}

Since $M$ is connected, $\Hol_p(\nabla^E)$ and $\Hol_q(\nabla^E)$ are conjugate as subgroups of $\GL(k,\bbR)$, if $k$ is the rank of $E$ and we have chosen identifications $E_p\simeq \bbR^k\simeq E_q$. We write $\Hol(\nabla^E)$ to mean this conjugacy class of subgroups of $\GL(k,\bbR)$. The following proposition is a very useful property of holonomy groups.

\begin{prop}
    \label{prop:ParallelSectionsActionOfHolonomy}
Let $M$ be a manifold, $E$ a vector bundle over $M$, and $\nabla^E$ a connection on $E$. Let $p\in M$ be a point. Then the parallel sections of $E$ are in one-to-one correspondence with the fixed points of the action of $\Hol_p(\nabla^E)$ on $E_p$.
\end{prop}

If $(M,g)$ is a Riemannian manifold we define the holonomy group of $g$, $\Hol(g)$, to be the holonomy group of the Levi-Civita connection of $(M,g)$.
%$\mathrm{Spin}(7)$ is one of the two exceptional holonomy groups on Berger's list of Riemannian holonomy groups, the other being $G_2$. More correctly the 8-dimensional spin representation of $\mathrm{Spin}(7)$ is an exceptional holonomy representation.
%
%
%We will briefly review the defini$\mathrm{Spin}(7)$ geometry in this section. We will define $\mathrm{Spin}(7)$ as a subgroup of $\GL(\bbR,8)$ and show that it contains $\mathrm{SU}(4)$ as a proper subgroup.
%
%
%%\begin{defn}
%%Let $(E,\nabla)$ be a vector bundle with connection on a manifold $M$. Let $p$ be a point in $M$. The holonomy group of $\nabla$ at $p$, $\Hol_p(\nabla)$, is the subgroup of $\GL(E_p)$ generated by parallel transport around loops based at $p$ using $\nabla$. 
%%\end{defn}
%%
%%\begin{defn}
%%Let $(M,g)$ be a Riemannian manifold and $\nabla^g$ the Levi-Civita connection on $TM$. Then we define the Riemannian holonomy group of $g$ to be $\Hol(g)=\Hol(\nabla^g)$
%%\end{defn}
%
%
Note that the holonomy group of a Riemannian manifold comes equipped with a representation on the fibres of the tangent bundle. Therefore when we say that a manifold has holonomy $\mathrm{Spin}(7)$ we must also say what representation of $\mathrm{Spin}(7)$ we are considering. 

$\mathrm{Spin}(7)$ can be defined as the simply-connected double cover of $SO(7)$.  We however will define it as the stabiliser group of a certain 4-form on $\bbR^8$, which will determine an embedding of $\mathrm{Spin}(7)$ in $\GL(8,\bbR)$ and hence the irreducible 8-dimensional representation, to which Berger's theorem refers.

\begin{defn}
Let $\bbR^8$ have coordinates $(x_1,\dotsc, x_8)$. Let $\rd x_{ijkl}$ denote the 4-form $\rd x_i\wedge \rd x_j\wedge \rd x_k\wedge \rd x_l$. We define the \emph{Cayley form}, $\Omega_0$, by
\begin{equation*}
\begin{split}
\Omega_0&=\rd x_{1234}+\rd x_{1256}+\rd x_{1278}+\rd x_{1357}-\rd x_{1368}-\rd x_{1458}-\rd x_{1467} \\
&+\rd x_{5678}+\rd x_{3478}+\rd x_{3456}+\rd x_{2468}-\rd x_{2457}-\rd x_{2367}-\rd x_{2358}
\end{split}
\end{equation*}
 $\mathrm{Spin}(7)$ is the subgroup of $\GL(8,\bbR)$ preserving $\Omega_0$.
\end{defn}

The 4-form above can be motivated by the structure of the octonions. The relationship between the octonions and the Cayley form can be found in, for example, \cite[Section IV.1.C]{Harvey:1982fk}. It should be noted that the Cayley form given above differs from that in \cite{Harvey:1982fk} by an orientation-preserving permutation of the coordinates and an overall change in sign.

Since we have defined $\mathrm{Spin}(7)$ as the stabilizer group of the Cayley form by Proposition \ref{prop:ParallelSectionsActionOfHolonomy}, if $(M,g)$ is an oriented Riemannian $8$-manifold with $\Hol(g)\subseteq \mathrm{Spin}(7)$, then $M$ admits a (not necessarily unique) parallel 4-form $\Omega$ such that for any $p\in M$ there exists an oriented isometry $T_pM\rightarrow \bbR^8$, which takes $\Omega_p$ to $\Omega_0$. 

We will define a $\mathrm{Spin}(7)$-manifold to include a choice of Cayley form. This fixes a particular embedding of $\Hol(g)\subseteq \mathrm{Spin}(7)$.

%Suppose $M$ is a connected manifold and $E$ is a vector bundle on $M$ with a connection $\nabla$.
%Recall that  parallel sections of $E$ are in one-to-one correspondence of fixed points of the action of the holonomy group of $\nabla$ on a generic fibre. 

% If $(M,g)$ is an oriented Riemannian $8$-manifold with $\Hol(g)\subseteq \mathrm{Spin}(7)$ then since we defined $\mathrm{Spin}(7)$ as the stabiliser of a 4-form we have that $M$ carries a parallel 4-form $\Omega$ such that for any $p\in M$ there exists an oriented isometry $T_pM\rightarrow \bbR^8$ which takes $\Omega_p$ to $\Omega_0$. 
 
% In fact the parallel 4-form is only determined up to sign by $\Hol(g)$. We define a $\mathrm{Spin}(7)$-manifold to include an explicit choice of Cayley form.
 
% We wish to consider manifolds with holonomy strictly contained in $\mathrm{Spin}(7)$ also. We therefore make the following definition.
 
 \begin{defn}
A \emph{$\mathrm{Spin}(7)$-manifold} is a triple $(M,\Omega,g)$ where $(M,g)$ is an oriented Riemannian $8$-manifold, $\Hol(g)\subseteq \mathrm{Spin}(7)$ and $\Omega$ is a parallel 4-form such that for any $p\in M$ there exists an oriented isometry $T_pM\rightarrow \bbR^8$, which takes $\Omega_p$ to $\Omega_0$. 
 \end{defn}
 
% Note that a Riemannian $8$-manifold $(M,g)$ with $\Hol(g)\subseteq \mathrm{Spin}(7)$ may be made into a $\mathrm{Spin}(7)$-manifold in more than one way. This is due to the fact that there may be inequivalent embeddings of $\Hol(g)\hookrightarrow \mathrm{Spin}(7)$. A $\mathrm{Spin}(7)$-manifold includes a choice of embedding. We will be interested in embeddings of $\mathrm{SU}(4)\hookrightarrow \mathrm{Spin}(7)$.
 
We can break up the condition of being a $\mathrm{Spin}(7)$-manifold into a topological one, namely the existence of a reduction of the structure group of $TM$ to $\mathrm{Spin}(7)$, and an integrability condition on this reduction.

\begin{defn}
Let $M$ be an oriented manifold $8$-manifold. A \emph{$\mathrm{Spin}(7)$-structure on $M$} is a pair $(\Omega,g)$ where $g$ is a Riemannian metric and for any $p\in M$ there exists an oriented isometry $T_pM\rightarrow \bbR^8$, which takes $\Omega_p$ to $\Omega_0$.
\end{defn}

A $\mathrm{Spin}(7)$-structure is equivalent to a reduction of the structure group of $TM$ to $\mathrm{Spin}(7)$.
The existence of a $\mathrm{Spin}(7)$-structure is a topological property of $M$ as the following result from \cite[Th. 10.7]{Lawson:1989fk} shows.

\begin{prop}
    Let $M$ be an oriented $8$-manifold. $M$ admits a $\emph{Spin}(7)$-structure  if and only if $w_2(M)=0$ and 
    \begin{equation*}
        p_1(M)^2-4p_2(M)+ 8\chi(M)=0.
    \end{equation*}
\end{prop}

For $M$ to be a $\mathrm{Spin}(7)$-manifold it must satisfy an extra integrability condition as the following proposition shows from \cite[Prop. 10.5.3]{joyce2000compact}.

\begin{prop}
    Let $M$ be an  oriented $8$-manifold with $\mathrm{Spin}(7)$-structure $(\Omega,g)$. Then $\Hol(g)\subseteq \mathrm{Spin}(7)$ and $\Omega$ is the induced $4$-form if and only if $\rd\Omega=0$. In this case we say the $\mathrm{Spin}(7)$-structure $(\Omega,g)$ is \emph{torsion-free}.
\end{prop}

The construction of Joyce uses a Calabi--Yau 4-orbifold to construct a $\mathrm{Spin}(7)$-manifold. Any Calabi--Yau 4-fold carries an $S^1$ of torsion-free $\mathrm{Spin}(7)$-structures as we will soon see. We will now define $\mathrm{SU}(4)$ as the stabiliser group of a set of tensors and show that $\mathrm{SU}(4)$ embeds into $\mathrm{Spin}(7)$.

$\mathrm{SU}(4)$ can be defined as the stabiliser of a metric, a K\"ahler form $\omega_0$ and holomorphic volume form $\theta_0$. If we let $(z_1,z_2,z_3,z_4)$ be coordinates on $\bbC^4$ we can write $\omega_0$, $ \theta_0$ as 
%\begin{align*}
%g_0&=|\rd z_1|^2+\dotsb +|\rd z_4|^2 & \omega_0&=\frac{\imi}{2}(\rd z_1\wedge d \overline{z}_1+\dotsb +\rd z_4\wedge d \overline{z}_4) \\
%&& \theta_0&=\rd z_1\wedge\dotsb \wedge \rd z_4
%\end{align*}
\begin{equation*}
 \omega_0=\frac{\imi}{2}(\rd z_1\wedge \rd \bar{z}_1+\dotsb +\rd z_4\wedge \rd \bar{z}_4) \text{ and }
 \theta_0=\rd z_1\wedge\dotsb \wedge \rd z_4.
\end{equation*}

We define a Calabi--Yau $4$-fold as a quadruple $(X,g,\omega,\theta)$ consisting of a K\"ahler manifold $(X,g,\omega)$ and a holomorphic $(4,0)$-form $\theta$ such that $|\theta|\equiv 4$. It can be shown that for any $p\in X$ there exists an isometry $T_pX\rightarrow \bbC^4$ taking $(\omega,\theta)$ to $(\omega_0,\theta_0)$.

\begin{prop}
\label{prop:cyisspin7}
Let $(X,g,\omega,\theta)$ be a Calabi--Yau $4$-fold.  Define a $4$-form by  $\Omega=\frac{1}{2}\omega\wedge\omega+\re{\theta}$, then $(\Omega,g)$ is a torsion-free $\mathrm{Spin}(7)$-structure on $X$.
\end{prop}
\begin{proof}
Let $p\in X$ and identify $\omega_p$ and $\theta_p$ with the standard forms on $\bbC^4$. Identifying $\bbC^4$ with $\bbR^8$ via $z_j=x_{2j-1}+\imi x_{2j}$ 
%$z_1=x_1+\imi x_2$, $z_2=x_3+\imi x_4$, $z_3=x_5+\imi x_6$ and $z_4=x_7+\imi x_8$ 
and comparing the expressions for $\Omega_p$ and $\Omega_0$ we see that $(\Omega,g)$ defines a $\mathrm{Spin}(7)$-structure. Since $(X,g,\omega,\theta)$ is a Calabi-Yau manifold we have $\rd\omega=\rd\theta=0$, which implies $\rd\Omega=0$.
\end{proof}

The proposition above describes a particular embedding of $\mathrm{SU}(4)\hookrightarrow \mathrm{Spin}(7)$. 
%We can make a Calabi--Yau 4-fold into a $\mathrm{Spin}(7)$-manifold in a number of ways. In fact we have 
%$\mathrm{Spin}(7)/\mathrm{SU}(4)\simeq \mathrm{Spin}(7)/\mathrm{Spin}(6)\simeq SO(7)/SO(6)\simeq S^6$ 
Let $(X,g,\omega,\theta)$ be a Calabi--Yau $4$-fold. Then the 4-form $\Omega_\phi=\frac{1}{2}\omega\wedge\omega+\re(e^{\imi\phi}\theta)$ for $\phi\in[0,2\pi)$ also defines a torsion-free $\mathrm{Spin}(7)$-structure on $X$ and a different embedding of $\mathrm{SU}(4)\hookrightarrow \mathrm{Spin}(7)$.

%In fact $\Omega=\pm\frac{1}{2}\omega\wedge\omega+\re(e^{\imi \phi}\theta)$ defines a torsion-free $\mathrm{Spin}(7)$-structure on $M$.

%
%Since we have an inclusion of Lie groups $\mathrm{SU}(4)\subset \mathrm{Spin}(7)$ it follows that any Riemannian manifold $(M,g)$ with $\Hol(g)\subset \mathrm{SU}(4)$ has an induced torsion-free $\mathrm{Spin}(7)$-structure. A manifold with $\Hol(g)\subset \mathrm{SU}(4)$ is a Calabi--Yau 4-fold, which has a holomorphic volume form as well as a K\"ahler form. We can describe the Cayley form in terms of these parallel forms as follows.
%

\section{Construction of $\mathrm{Spin}(7)$-manifolds}
The essential idea in Joyce's constructions of manifolds with exceptional holonomy is that of resolving the singularities of orbifolds within a particular holonomy group. We will therefore review the definitions of orbifolds and discuss Riemannian metrics and their holonomy groups on orbifolds. We will then give a short overview of the construction of manifolds with holonomy $\mathrm{Spin}(7)$ from Calabi--Yau 4-orbifolds. We direct the reader to \cite{Joyce:1999fk} and \cite[Ch. 10]{joyce2000compact} for the details of the construction.

\subsection{Orbifolds}
\begin{defn}
An \emph{orbifold} is a singular manifold $M$ of dimension $n$ whose singularities are locally isomorphic to quotient singularities $\bbR^n/G$ for finite subgroups $G\subset \GL(n,\bbR)$, such that if $1\neq \gamma\in G$, then the subspace $V_\gamma$ of $\bbR^n$ fixed by $\gamma$ has $\dim V_\gamma\leq n-2$.
\end{defn}

We say a point $p$ in $M$ is an \emph{orbifold point with orbifold group $G$} if $M$ is locally isomorphic to $\bbR^n/G$ at $p$ with $G$ non-trivial. 

%There is a good notion of a Riemannian metric on an orbifold. 

\begin{defn}
A \emph{Riemannian metric $g$ on an orbifold $M$} is a Riemannian metric in the usual sense on the nonsingular part of $M$ and where $M$ is locally isomorphic to $\bbR^n/G$, the metric $g$ can be identified with the quotient of a $G$-invariant Riemannian metric defined on an open set of $0$ in $\bbR^n$.
We define the \emph{holonomy group} $\Hol(g)$ of $g$ to be the holonomy group of the restriction of $g$ to the nonsingular part of $M$.
\end{defn}

If $p$ is an orbifold point of $M$ with orbifold group $G$ then we have an inclusion of groups
\[
G\subseteq \Hol(g).
\]
Therefore for an orbifold to have holonomy $\mathrm{Spin}(7)$ we must have that each orbifold group $G$ lies in (the conjugacy class of subgroups of) $\mathrm{Spin}(7)$.

Many results for manifolds carry over with small modifications to orbifolds. In particular the Calabi conjecture holds for compact K\"ahler orbifolds. 
As a consequence we have the following theorem \cite[Th. 6.5.6]{joyce2000compact}, which we can use  to find Calabi--Yau metrics on orbifolds.

\begin{thm}
\label{thm:ccorbifolds}
Let $X$ be a compact complex orbifold with $c_1(X)=0$ admitting K\"ahler metrics. Then there is a unique Ricci-flat K\"ahler metric in every K\"ahler class on $X$.
\end{thm}

%The main ingredient of the construction of manifolds with holonomy $\mathrm{Spin}(7)$ is a Calabi--Yau 4-fold, $Y$. 
%Let $Y$ be a Calabi--Yau 4-fold.
%By Proposition \ref{prop:cyisspin7}, $Y$ admits a torsion-free $\mathrm{Spin}(7)$-structure. However this $\mathrm{Spin}(7)$-structure has holonomy strictly contained in $\mathrm{Spin}(7)$. A good idea is to divide $Y$ by a finite group of isometries which preserves the $\mathrm{Spin}(7)$-structure but does not preserve the $\mathrm{SU}(4)$-structure and then resolve the singularities arising from the fixed points of the action.

%Summary of general idea to resolve orbifold singularities using ALE spaces with special holonomy.

\subsection{Resolution of singularities}
Given a Riemannian orbifold $(M,g)$ with holonomy $\Hol(g)$ we would like to know whether we can find a resolution $(\hat{M},\hat{g})$ of $M$ such that $\Hol(\hat{g})\subseteq \Hol(g)$.

If $Hol(g)\subset \mathrm{SU}(n)$ and then we can use complex geometry to determine whether resolutions exist which do not change the holonomy group.
Suppose $X$ is a Gorenstein algebraic variety so the canonical sheaf $\mathcal{O}(K_X)$ in invertible. 
A resolution $\pi:\hat{X}\rightarrow X$ is called \emph{crepant} if $\pi^*(\mathcal{O}(K_X))=\mathcal{O}(K_{\hat{X}})$.
The importance of crepant resolutions is that a crepant resolution of a Calabi--Yau orbifold is a Calabi--Yau manifold. 
The question of  existence and uniqueness of crepant resolutions of quotient singularities $\bbC^n/G$ for finite groups $G\subset \mathrm{SU}(n)$ is a difficult one, to which we will return later.

If we wish to construct manifolds with special holonomy then we will not, in general, be able to use algebraic techniques and instead we must rely on finding singularities of a type, which we know we can resolve within a given holonomy group.

\subsection{Review of construction}
\label{sec:ReviewofConstruction}
Now suppose $Y$ is a complex $4$-orbifold admitting metrics with holonomy $\mathrm{SU}(4)$.
We require $Y$ to have isolated singularities $\{p_1,\dotsc,p_k\}$ modelled on $\bbC^4/\bbZ_4$, where the generator of $\bbZ_4$ acts as
\begin{equation}
    \label{eq:alpha}
 \alpha:(z_1,z_2,z_3,z_4)\mapsto (iz_1,iz_2,iz_3,iz_4).
\end{equation}
The group $\bbZ_4$ lies in $\mathrm{SU}(4)$, which is consistent with $Y$ having holonomy $\mathrm{SU}(4)$.  We also require that $Y$ admits an antiholomorphic isometric involution $\tau$ with fixed points the finite set $\{p_1,\dotsc,p_k\}$. 

Since $\tau$ is antiholomorphic, the complex structure does not descend to the quotient of $Y$ by $\tau$. However we can form a torsion-free $\mathrm{Spin}(7)$-structure $(\Omega,g)$ given by $\Omega=\frac{1}{2}\omega\wedge\omega+\re{\theta}$, which is $\tau$-invariant.

Defining  $Z=Y/\langle \tau\rangle$ we have  that this $\tau$-invariant torsion-free $\mathrm{Spin}(7)$-structure on $Y$ descends to $Z$ and the orbifold singularities of $Z$ are modelled on $\bbR^8/G$ where $G$ is a finite subgroup of $\mathrm{Spin}(7)$. 
We now wish to resolve these singularities by glueing in ALE $\mathrm{Spin}(7)$-manifold to construct a $\mathrm{Spin}(7)$-manifold $M$.

%It is shown in \cite[Prop. 5.3]{Joyce:1999fk} that we can choose an isometry $T_{p_j}Y\rightarrow \bbC^4$ taking $g$, $\omega$ and $\theta$ to their standard forms on $\bbC^4$ and the action of $\tau_{p_j}$ on $T_{p_j}Y$ is  identified with the map $\beta:\bbC^4\rightarrow \bbC^4$ given by
%  \[
%  \beta:(z_1,z_2,z_3,z_4)\mapsto (\overline{z}_2,-\overline{z}_1,\overline{z}_4,-\overline{z}_3)
%  \]

%The essential feature/part/idea of the construction is that the singularity $\bbR^8/G$ can be resolved with holonomy contained in $\mathrm{Spin}(7)$ in more than one way. The resolutions $X_1$ and $X_2$ have holonomy  $\bbZ_2\ltimes   \mathrm{SU}(4)$ but the embeddings of $\bbZ_2\ltimes  \mathrm{SU}(4)\hookrightarrow \mathrm{Spin}(7)$ are different.

It is shown in \cite[Prop. 5.3]{Joyce:1999fk} that all the singularities are of the same form.
If we define coordinates on $\bbR^8$ as $(x_1,\dotsc,x_8)$ and complex coordinates by $z_i=x_{2i-1}+i x_{2i}$ then the singularities are of the form $\bbR^8/G$ where $G=\langle\alpha,\beta\rangle$ is a finite non-abelian subgroup of $\mathrm{Spin}(7)$ generated by $\alpha$, which is described in Eq. \eqref{eq:alpha}, and $\beta$, whose action on $\bbC^4$ is given by 
\begin{equation*}
    \beta:(z_1,z_2,z_3,z_4)\mapsto (\overline{z}_2,-\overline{z}_1,\overline{z}_4,-\overline{z}_3).
\end{equation*}

The singularity $\bbR^8/G$ can be resolved with holonomy contained in $\mathrm{Spin}(7)$ in two ways. Both resolutions $X_1,X_2$ have the same holonomy as abstract Lie groups but the embeddings into $\GL(8,\bbR)$ (or $\mathrm{Spin}(7)$) are different. 
If we choose the resolutions of the singularities of $Z$ wisely we can ensure that the holonomy of the resolution of the orbifold $Z$ has holonomy exactly $\mathrm{Spin}(7)$ and not a proper subgroup of it.

%
%\subsection{Review of Dominic's Construction}
%
The precise necessary conditions on the complex 4-orbifold, $Y$, are stated below. 

%We will now review the necessary conditions on the Calabi--Yau 4-fold for the construction to work which is taken from [ref,DJ].

\begin{cond}
Let $Y$ be a compact complex 4-orbifold with $c_1(Y)=0$, admitting K\"ahler metrics.
Let $\tau$ be an antiholomorphic involution on $Y$.
We require that $Y$ have isolated singularities $\{p_1\dotsc,p_k\}$, with $k\geq 1$, modelled on $\bbC^4/\bbZ_4$ as described above and that the fixed point set of $\tau$ is $\{p_1,\dotsc,p_k\}$. We also require that $Y\setminus \{p_1,\dotsc,p_k\}$ is simply-connected and $h^{2,0}(Y)=0$.
\label{cond:Y}
\end{cond}

We will use the following theorem of Joyce to construct examples of $\mathrm{Spin}(7)$-manifolds from appropriate complex $4$-orbifolds, which can be found in \cite[Th. 5.14]{Joyce:1999fk}.

\begin{thm}
Suppose $Y$ satisfies Condition \ref{cond:Y}.
Let $M$ be the resulting compact $8$-manifold defined in \cite[Def. 5.8]{Joyce:1999fk}. 
Then there exist torsion-free $\mathrm{Spin}(7)$-structures $(\Omega,g)$ on $M$. 
We can choose the resolutions of the singularities so that $\Hol(g)=\mathrm{Spin}(7)$.
\end{thm}

%Due to the Calabi conjecture, for every K\"ahler class on $Y$ we have a Calabi--Yau metric. Then by the discussion in section [ref,section above] we have a torsion-free $\mathrm{Spin}(7)$-structure $(\Omega,g)$ on $Y$. The K\"ahler class can be chosen to be $\tau$-invariant and hence the $\mathrm{Spin}(7)$-structure can be chosen to be $\tau$-invariant. This then descends to a torsion-free $\mathrm{Spin}(7)$-structure on the quotient $Z=Y/\langle \tau\rangle$. 

%By theorem [ref,DJ page 400] we can find a resolution $M$ of $Z$ which carries a torsion-free $\mathrm{Spin}(7)$-structure $(\tilde{\Omega},\tilde{g})$ with $\Hol(\tilde{g})=\mathrm{Spin}(7)$.

\section{Weighted projective spaces and hypersurfaces}

Our goal is now to find complex 4-orbifolds which satisfy Condition \ref{cond:Y}. We will use algebraic geometry to find examples of such orbifolds. Hypersurfaces in weighted projective spaces provide a large source of orbifolds with specified (cyclic quotient) singularities.
We will therefore begin by reviewing weighted projective spaces, their singularities and hypersurfaces contained in weighted projective spaces. 
The majority of this section is from \cite{ianofletcher}. 

\begin{defn}
Let $a_0,\dotsc,a_n$ be positive integers with $\gcd(a_0,\dotsc,a_n)=1$. The \emph{weighted projective space} $\bbC\bbP^n_{a_0,\dotsc,a_n}$ with weights $a_0,\dotsc,a_n$ is the quotient of $\bbC^{n+1}\setminus\{0\}$ by the action of $\bbC^*$ given by
\[
\lambda:(z_0,\dotsc,z_n)\mapsto (\lambda^{a_0}z_0,\dotsc,\lambda^{a_n}z_n).
\]
\end{defn}

In general $\bbC\bbP^n_{a_0,\dotsc,a_n}$ will have singularities where the action of $\bbC^*$ on $\bbC^{n+1}\setminus\{0\}$ is not free.
The stabiliser groups of these points are finite and so we can treat $\bbC\bbP^n_{a_0,\dotsc,a_n}$ as an orbifold.
We can also treat $\bbC\bbP^n_{a_0,\dotsc,a_n}$ as a singular algebraic variety by considering $\bbC\bbP^n_{a_0,\dotsc,a_n}$ as $\Proj$ of a graded ring.
This will be a useful viewpoint for us because it shows the similarities between weighted projective spaces and the usual straight projective space.

Let $R$ be the graded ring $\bbC[z_0,\dotsc,z_n]$ where $z_i$ has weight $a_i$. $R$ has a direct sum decomposition $R=\bigoplus_dR_d$ into its graded pieces. Elements of $R_d$ will be called \emph{weighted homogeneous polynomials of degree $d$} but we will soon drop the term weighted and leave it as understood.

We can treat $\bbC\bbP^n_{a_0,\dotsc,a_n}$ as a variety as $\Proj(R)$. From generalities on taking $\Proj$ of graded rings, see \cite[Prop. 5.11]{Hartshorne:1977fk}, we have that a finitely generated graded $R$-module determines a coherent sheaf of $\mathcal{O}_{\bbC\bbP^n_{a_0,\dotsc,a_n}}$-modules.
In particular the module $R(m)$ determines a sheaf, which we will denote $\mathcal{O}(m)$ for brevity.

%We can treat $\bbC\bbP^n_{a_0,\dotsc,a_n}$ as an orbifold or as a singular algebraic variety. 
%[ref,Reid graded rings] describes the viewpoint of $\bbC\bbP^n_{a_0,\dotsc,a_n}$ as an algebraic variety. 
We should be careful when distinguishing between the two viewpoints. For example we have the following result from \cite[Cor. 5.9]{ianofletcher}, which holds only when considering $\bbC\bbP^n_{a_0,\dotsc,a_n}$ as a variety.

\begin{lem}
Let $a_0,\dotsc,a_n$ be positive integers with $\gcd(a_0,\dotsc,a_n)=1$. Let $q=\gcd(a_1,\dotsc,a_n)$. Then $\bbC\bbP^n_{a_0,\dotsc,a_n}\simeq \bbC\bbP^n_{a_0,a_1/q,\dotsc,a_n/q}$ as varieties.
\end{lem}

We have the following corollary.

\begin{cor}
Let $a_0,\dotsc,a_n$ be positive integers with $\gcd(a_0,\dotsc,a_n)=1$. Then $\bbC\bbP^n_{a_0,\dotsc,a_n}\simeq \bbC\bbP^n_{b_0,\dotsc,b_n}$ as varieties for some weights $b_0,\dotsc,b_n$ such that 
$
\gcd(b_0,\dotsc,b_{i-1},b_{i+1},\dotsc,b_n)=1
$
 for each $i$.
\end{cor}

This motivates the following definition.

\begin{defn}
We say 
$\bbC\bbP^n_{a_0,\dotsc,a_n}$
is \emph{well-formed} if 
\[
\gcd(a_0,\dotsc,a_{i-1},a_{i+1},\dotsc,a_n)=1 \text{ for each } i.
\]
\end{defn}

The condition of being well-formed is related to the structure of the singularities of $\bbC\bbP^n_{a_0,\dotsc,a_n}$. 
The singularities of $\bbC\bbP^n_{a_0,\dotsc,a_n}$ are all cyclic quotient singularities. 
We say a  cyclic quotient singularity $\bbC^n/\bbZ_m$ is of type $\frac{1}{m}(a_1,\dotsc,a_n)$ if $\bbZ_m$ acts on $\bbC^n$ as
\[
(z_1,\dotsc,z_n)\xrightarrow{\xi}(\xi^{a_1}z_1,\dotsc,\xi^{a_n}z_n)
\]
where $\xi^m=1$.

For any subset $I\subset \{0,\dotsc,n\}$ we define $S_I=\{[z_0,\dotsc,z_n]:z_{j}=0;\text{ $\forall j\notin I$}\}\allowbreak\subset \bbC\bbP^n_{a_0,\dotsc,a_n} $.
Now suppose $\gcd(a_{i_0},\dotsc,a_{i_k})=m\neq 1$, then a generic point $p \in S_{i_0,\dotsc,i_k}$ is an orbifold point modelled on the singularity $\bbC^k\times \bbC^{n-k}/\bbZ_m$.
If we extend the sequence $(i_0,\dotsc,i_k)$ to be a permutation $(i_0,\dotsc,i_n)$ of the sequence  $(0,\dotsc,n)$ then the singularity $\bbC^{n-k}/\bbZ_m$ is of type $\frac{1}{m}(a_{i_{k+1}},\dotsc,a_{i_n})$. 

%action of $\bbZ_m$ on $\bbC^{n-k}$ is given by
%\[
%\zeta:(z_{i_{k+1}},\dotsc,z_{i_{n}})\mapsto (\zeta^{a_{i_{k+1}}}z_{i_{k+1}},\dotsc,\zeta^{a_{i_n}} z_{i_n})
%\]
%where $\zeta^m=1$.

%\begin{exmp}
%Consider the weighted projective space $\bbC\bbP^5_{1,1,21,21,12,28}$. The singular set consists of the union of two surfaces $S_{2,3,4}$, $S_{2,3,5}$, which intersect along the curve $S_{2,3}$, and the curve $S_{4,5}$. The singularity at a generic point of $S_{2,3,4}$ is modelled on $\bbC^2\times \bbC^3/\bbZ_3$ where $\bbZ_3$ acts as 
%\end{exmp}
%
%\begin{exmp}
%Consider the weighted projective space $\bbC\bbP^3_{1,2,3,6}$. The singular locus consists of the union of the two curves $S_{\{1,3\}}\cup S_{\{2,3\}}$ which intersect at the singular point $S_{\{3\}}$. The singularity at a generic point in $S_{\{1,3\}}$ is modelled on $\bbC\times \bbC^2/\bbZ_2$ where $\bbZ_2$ acts as $(z_0,z_2)\mapsto(-z_0,-z_2)$ and the singularity at a generic point in $S_{\{1,3\}}$ is modelled on $\bbC\times\bbC^2/\bbZ_3$ where $\bbZ_3$ acts as $(z_0,z_1)\mapsto (\zeta_3z_0,\zeta_3^{-1}z_1)$ where $\zeta_3^3=1$. Finally we have a nonisolated singular point at $S_{\{3\}}$ which is modelled on $\bbC^3/\bbZ_6$ where $\bbZ_6$ acts as $(z_0,z_1,z_2)\mapsto (\zeta_6z_0,\zeta_6^2z_1,\zeta_6^3z_2)$.
%
%\end{exmp}
\begin{exmp}
Consider the weighted projective space $\bbC\bbP^3_{1,2,3,6}$. The singular locus consists of the union of the two curves $S_{1,3}\cup S_{2,3}$, which intersect at the singular point $S_{3}$. The singularity at a generic point of $S_{1,3}$ is modelled on $\bbC\times \bbC^2/\bbZ_2$, where $\bbZ_2$ acts as $(z_0,z_2)\mapsto(-z_0,-z_2)$. The singularity at a generic point of $S_{1,3}$ is modelled on $\bbC\times\bbC^2/\bbZ_3$, where $\bbZ_3$ acts as $(z_0,z_1)\mapsto (\xi z_0,\xi^{-1}z_1)$ and $\xi^3=1$. Finally we have a nonisolated singular point at $S_{3}$, which is modelled on $\bbC^3/\bbZ_6$, where $\bbZ_6$ acts as $(z_0,z_1,z_2)\mapsto (\zeta z_0,\zeta^2z_1,\zeta^3z_2)$ and $\zeta^6=1$.

\end{exmp}

From the description of singularities above we see that the condition of being well-formed is equivalent to $\bbC\bbP^n_{a_0,\dotsc,a_n}$ having only singularities in complex codimension greater than 1.
% We can think of weighted projective spaces as being natural generalisations of projective space to orbifolds.

\subsection{Hypersurfaces}

A section $f\in \Gamma(\bbC\bbP^n_{a_0,\dotsc,a_n},\mathcal{O}(d))=R_d$ determines a hypersurface in $\bbC\bbP^n_{a_0,\dotsc,a_n}$, by definition of degree $d$.  
It can be shown that any hypersurface is determined by such a section \cite[Th. 3.7]{Cox:1995uq}. 

%$\bbC\bbP^n_{a_0,\dotsc,a_n}$ can be described as an algebraic variety as $\Proj$ of a graded ring. This graded ring then allows us to treat hypersurfaces of weighted projective space in a similar way to straight projective space.

%Consider the graded ring $R=\bbC[z_0,\dotsc,z_n]$ where $z_i$ has weight $a_i$. $R$ has a direct sum decomposition into its graded parts $R=\bigoplus_d R_d$ where $R_d=\bbC\langle z_0^{d_0}\dotsm z_n^{d_n}: \sum_id_ia_i=d\rangle$. $\bbC^*$ acts on $R_d$ with weight $d$. Then as promised we can make $\bbC\bbP^n_{a_0,\dotsc,a_n}$ into a variety by taking $\Proj(R)$.

%We call a polynomial $f\in R_d$ a \emph{(weighted) homogeneous polynomial of degree $d$}. We will soon drop the term weighted and leave it as understood. Just as in the case of straight projective space, the zero set of  a weighted homogeneous polynomial defines a hypersurface in $\bbC\bbP^n_{a_0,\dotsc,a_n}$. 

%Also the hypersurface $Y=\{f=0\}$ is quasismooth iff $\{df=0\}=\emptyset$ in $\bbC^{n+1}\setminus \{0\}$.

%We wish to find hypersurfaces in $\bbC\bbP^n_{a_0,\dotsc,a_n}$ with particular kinds of singularities. To begin we should first understand the singularities arising from $\bbC\bbP^n_{a_0,\dotsc,a_n}$ itself.

\subsubsection{Quasismoothness}

%We now wish to define the notion of quasismoothness. 
Recall that a projective variety is smooth if the affine cone is smooth away from the origin. We shall define quasismoothness in a similar way. Let $\pi:\bbC^{n+1}\setminus\{0\}\mapsto \bbC\bbP^n_{a_0,\dotsc,a_n}$ be the projection.

%Let $Y\subset\bbC\bbP^n_{a_0,\dotsc,a_n}$ be an algebraic variety. We have a quotient map $\pi:\bbC^{n+1}\setminus \{0\}\mapsto \bbC\bbP^n_{a_0,\dotsc,a_n}$ and as in the case for projective space we can define the affine quasi-cone of $Y$, $C(Y)=\pi^{-1}(Y)\cup \{0\}$.
 
 \begin{defn}
 Let $Y\subset \bbC\bbP^n_{a_0,\dotsc,a_n}$ be an algebraic variety. We say $Y$ is \emph{quasismooth} if $\pi^{-1}(Y)$ is smooth.
 \end{defn}
 
An algebraic variety $Y$ is quasismooth if $Y$ only has singularities coming from the orbifold singularities of $\bbC\bbP_{a_0,\dotsc,a_n}$. We will restrict our attention to quasismooth hypersurfaces because we can understand their singularities easily in terms of those of the ambient weighted projective space. Regarding $\bbC\bbP^n_{a_0,\dotsc,a_n}$ as an orbifold, quasismoothness of $Y$ is equivalent to being a sub-orbifold, see \cite[Prop. 3.5]{batyrev1994hodge}.

%Note that we cannot always find quasismooth hypersurfaces of a fixed degree in a particular weighted projective space. For example there are no hypersurfaces of degree 1 in $\bbC\bbP^2_{2,3}$.

Note that a generic hypersurface of a fixed degree in a particular weighted projective space is not necessarily quasismooth.

\begin{exmp}
Consider the graded ring $R=\bbC[z_0,z_1,z_2]$ where the weights are $1,2,2$ respectively, then $\Proj(R)=\bbC\bbP^2_{1,2,2}$. A generic weighted homogeneous polynomial of degree $3$ is of the form $f=\lambda_1z_0z_1+\lambda_2z_0z_2+\lambda_3z_0^3$. The hypersurface $Y_3=V(f)$ is not quasismooth since the affine variety defined by $f$ is singular along the set $\{(z_0,z_1,z_2)\in \bbC^3:\text{ $z_0=0$ and $\lambda_1 z_1+\lambda_2 z_2=0$}\}$.
\end{exmp}

The conditions for the generic hypersurface of degree $d$ to be quasismooth are described below, taken from \cite[Th. 8.1]{ianofletcher}.

\begin{thm}
The generic hypersurface $Y_d$ of degree d in $\bbC\bbP^n_{a_0,\dotsc,a_n}$ is quasismooth if and only if either $a_i=d$ for some $i$, i.e. $Y_d$ is a linear cone, or for every nonempty subset $\{i_0,\dotsc,i_k\}\subset \{0,\dotsc,n\}$ either
\begin{enumerate}
\item there exists a monomial $z_{i_0}^{d_0}\dotsm z_{i_k}^{d_k}$ of degree $d$; or
\item for $j=0,\dotsc,k$ there exist monomials $z_{i_0}^{d_{0,j}}\dotsm z_{i_k}^{d_{k,j}} z_{e_j}$ of degree $d$, where the $e_j\notin\{i_0,\dotsc,i_k\}$ are distinct.
\end{enumerate}
\end{thm}

\subsubsection{Canonical Sheaf of a Hypersurface}

Recall that if $Y_d\subset \bbC\bbP^n$ is a smooth hypersurface of degree $d$, i.e. defined by a homogeneous polynomial of degree $d$, then the adjunction formula gives us that $K_{Y_d}=\mathcal{O}(d-n-1)|_{Y_d}$. 
We would like a similar result for weighted projective spaces so that we could test the triviality of the canonical sheaf easily.
%Let $Y_d$ be a generic hypersurface of degree $d$. We would like to have an adjunction formula so that we could test the triviality of the canonical bundle easily.
%We now want to have an easy test to see if $c_1(Y)=0$ so we want a result like the adjunction formula.
 Fortunately we have such a result for a large class of hypersurfaces, namely those which are quasismooth and well-formed.
 %In order for the adjunction formula to hold we need a restriction on the subvariety $Y$ as described in [ref, Iano-Fletcher]. 

\begin{defn}
Let $Y\subset \bbC\bbP^n_{a_0,\dotsc,a_n}$ be a hypersurface.
 We say  $Y$ is \emph{well-formed} if $\bbC\bbP^n_{a_0,\dotsc,a_n}$ is well-formed and $Y$ does not contain a codimension 2 singular set of $\bbC\bbP_{a_0,\dotsc,a_n}$. 
 \end{defn}

We have the following criterion for well-formedness for generic hypersurfaces from \cite[Prop. 6.10]{ianofletcher}.
\begin{prop}
  \label{prop:GenericHypersurfaceWellFormed}
The generic hypersurface of degree $d$ in $\bbC\bbP^n_{a_0,\dotsc,a_n}$ is well-formed if and only if  
\begin{enumerate}
\item $\gcd(a_0,\dotsc,a_{i-1},a_{i+1},\dotsc,a_{j-1},a_{j+1},\dotsc,a_n)\mid d$ for all $i,j$ and
\item $\gcd(a_0,\dotsc,a_{i-1},a_{i+1},\dotsc,a_n)=1$ for all $i$.
\end{enumerate}
\end{prop} 
  
   \begin{prop}
       Let $Y_d\subset\bbC\bbP^n_{a_0,\dotsc,a_n}$ be a well-formed quasismooth hypersurface of degree $d$. Then the canonical sheaf  is $K_{Y_d}=\mathcal{O}(d-n-1)|_{Y_d}$.
  \end{prop}

Hence for $n>1$ we have that the canonical bundle of a degree $d$ well-formed quasismooth hypersurface is trivial if $d=n+1$.

\section{Antiholomorphic involutions}
Let $Y\subset \bbC\bbP^5_{a_0,\dotsc,a_5}$ be a well-formed quasismooth hypersurface with trivial canonical bundle. We now wish to consider antiholomorphic involutions on $Y$.
We will consider only antiholomorphic involutions which arise as restrictions of antiholomorphic involutions on  $\bbC\bbP^5_{a_0,\dotsc,a_5}$.
The main result of this section is the classification of antiholomorphic involutions of weighted projective spaces, Proposition \ref{prop:antiholinv}.

It is shown in \cite{partouche2001rolling} that for standard projective space $\bbC\bbP^n$ the number of conjugacy classes of antiholomorphic involutions is either 1 or 2 depending on whether $n$ is odd or even respectively. If $n$ is odd, then the only antiholomorphic involution of $\bbC\bbP^n$ up to conjugacy is the standard one:
\[
[z_0,\dotsc,z_n]\longmapsto [\overline{z}_0,\dotsc,\overline{z}_n].
\]
If $n$ is even we also have the involution
\[
[z_0,\dotsc,z_n]\longmapsto [\overline{z}_1,-\overline{z}_0,\dotsc,\overline{z}_n,-\overline{z}_{n-1}].
\]

We will consider antiholomorphic involutions up to conjugation by automorphisms of $\bbC\bbP^n_{a_0,\dotsc,a_n}$. Therefore we should first describe $\Aut(\bbC\bbP^n_{a_0,\dotsc,a_n})$.

\subsection{Automorphisms of weighted projective spaces}
Consider the action of $\bbC^*$ on $\bbC^{n+1}$ defining a weighted projective space with weights $a_0,\dotsc,a_n$.
We want to decompose $\bbC^{n+1}$ by the action of $\bbC^*$. 
We relabel the collection of weights $w_1<\dotsb <w_m$ and let $k_i$ be the number of times $w_i$ appears in the sequence $a_0,\dotsc,a_n$. We decompose $\bbC^{n+1}=\bigoplus_{i\in I}W_i$ where $\bbC^*$ acts on $W_i$ with weight $w_i$, $\dim(W_i)=k_i$ and $I=\{1,\dotsc,m\}$.

%Let $I=\{1,\dotsc,m\}$ and $W=\bigoplus_{i\in I}W_i$ be a complex vector space with $\dim W_i=k_i$. Consider the action of $\bbC^*$ on $W$ given by letting $\bbC^*$ act on $W_i$ with weight $w_i$. The quotient of $W$ by this action is exactly $\bbC\bbP^n_{a_0,\dotsc,a_n}$.

Any automorphism of $\bbC^{n+1}\setminus \{0\}$ extends to an automorphism of $\bbC^{n+1}$ for $n>0$ by Hartogs' Theorem.
Let $\Aut_{\bbC^*}(\bbC^{n+1})$ denote the $\bbC^*$-equivariant automorphisms of $\bbC^{n+1}$. 
We can also describe $\Aut_{\bbC^*}(\bbC^{n+1})$ as the centralizer of $\bbC^*$ in $\Aut(\bbC^{n+1})$. 
Any $\bbC^*$-equivariant morphism of $\bbC^{n+1}$ descends to an automorphism of $\bbC\bbP^n_{a_0,\dotsc,a_n}$ and the converse can be shown to hold.
%normalizer of $\bbC^*$ in $\Aut(\bbC^{n+1})$. This is equal to the centralizer of $\bbC^*$, i.e. any automorphism of $\bbC\bbP^n_{a_0,\dotsc,a_n}$ lifts to a $\bbC^*$-equivariant morphism of $\bbC^{n+1}$.

A $\bbC^*$-equivariant morphism $F:\bbC^{n+1}\rightarrow \bbC^{n+1}$ is determined by a collection of polynomials $(F_{i,j})$ where $i\in I$, $1\leq j\leq k_i$ and $F_{i,j}$ is of degree $w_i$. 
Each polynomial $F_{i,j}$ can be decomposed 
\[
    F_{i,j}=A_{i,j}+f_{i,j}
\]
into a linear part and a non-linear, i.e. weighted homogenous quadratic and higher, part. 

\begin{exmp}
Consider the graded ring $R=\bbC[z_0,z_1,z_2]$ where $z_0,z_1,z_2$ have weights $1,1,2$ respectively. 
$\bbC^3$ splits as $\bbC^3=W_1\oplus W_2$ as representations of $\bbC^*$ where $\bbC^*$ acts on $W_1$ with weight 1 and on $W_2$ with weight 2.

A $\bbC^*$-equivariant morphism $F:\bbC^3\rightarrow\bbC^3$ is determined by a collection $F_{1,1},F_{1,2},F_{2}$ where $F_{1,1},F_{1,2}$ are linear functions of $z_0,z_1$ and $F_2$ is a sum of a linear multiple of $z_2$ and a homogeneous quadratic polynomial in $z_0,z_1$.
\end{exmp}

For each $i\in I$ we define $A_i$ to be the matrix formed from the rows $(A_{i,j})_{1\leq j\leq k_i}$.
The morphism $F$ is invertible on $W_1$ if $A_1$ is invertible since there are no polynomials with degree less than $w_1$.
An inductive argument gives us that $F$ is invertible if and only if each $A_i$ is.

The map that sends a morphism $F$ to the corresponding collection of linear maps $A_i$ is a surjective homomorphism
\[
\Aut_{\bbC^*}(\bbC^{n+1})\longrightarrow \prod_{i\in I} \GL(W_i).
\]

The kernel of this homomorphism is the set of morphisms of the form $F_{i,j}=z_{i,j}+f_{i,j}$ where $(z_{i,j})_{1\leq j\leq k_i}$ are coordinates on $W_i$. Let us denote this kernel by $H$. We have an inclusion of groups $\prod_{i\in I}\GL(W_i)\hookrightarrow \Aut_{\bbC^*}(\bbC^{n+1})$ and hence the short exact sequence
\[
0\longrightarrow H\longrightarrow \Aut_{\bbC^*}(\bbC^{n+1})\longrightarrow \prod_{i\in I} \GL(W_i)\longrightarrow 0
\]
is right split and we have $\Aut_{\bbC^*}(\bbC^{n+1})\simeq H\rtimes \prod_{i\in I}\GL(W_i)$.

Each element of $\Aut_{\bbC^*}(\bbC^{n+1})$ determines an automorphism of $\bbC\bbP^n_{a_0,\dotsc,a_n}$, but in order to determine an automorphism uniquely we must take a quotient by the diagonal action of $\bbC^*$ on $\Aut_{\bbC^*}(\bbC^{n+1})$.
%We have to projectivize $\Aut_{\bbC^*}(W)$ to get $\Aut(\bbC\bbP^n_{a_0,\dotsc,a_n})$. 
%The kernel of the homomorphism
%\[
%H\rtimes\prod_{i\in I}\GL(W_i)\rightarrow \Aut(\bbC\bbP^n_{a_0,\dotsc,a_n})
%\]
%is a copy of $\bbC^*$ embedded as
More explicitly, consider the homomorphism
\begin{align*}
\bbC^*&\hookrightarrow \GL(W_1)\times\dotsb\times \GL(W_m) \\
\lambda&\mapsto (\lambda^{w_1},\lambda^{w_2},\dotsc,\lambda^{w_m})
\end{align*}
which is an embedding since $\gcd(w_1,\dotsc,w_m)=1$. Then the quotient of $\Aut_{\bbC^*}(\bbC^{n+1})$ by this subgroup is isomorphic to $\Aut(\bbC\bbP^n_{a_0,\dotsc,a_n})$.

The linear structure on polynomials gives a linear structure to the group $H$.
The group $\prod_{i\in I}\GL(W_i)$ acts on  $H$ via the adjoint action, which we will denote $\Ad_A:H\rightarrow H$ for $A\in \prod_{i\in I}\GL(W_i)$.
With respect to this linear structure the adjoint action of $\prod_{i \in I}\GL(W_i)$ on $H$ is linear. 
$H$ decomposes as a vector space as $H=\bigoplus_{i\in I} H_i$ where $H_i$ consists of the morphisms in $H$ with $f_{i^\prime,j}=0$ for $i^\prime\neq i$.

We can describe some of the group structure on $H$ using the order on $I$ given by $i<i^\prime$ if $w_i<w_{i^\prime}$. 
%We define a partial order on the set $I$ by $i\leq i^\prime$ if $\exists d_1,\dotsc,d_m\geq 0$ such that $d_i\neq 0,d_{i^\prime}=0$ and $\sum_{l\neq i^\prime} d_l w_l=w_{i^\prime}$. Or in other words $i\leq i^\prime$ if there exists a monomial of weighted degree $w_{i^\prime}$ containing a variable of weighted degree $w_i$. 
For $f,g\in H$, let $i$ be such that $f_{i^\prime,j}=0$ $\forall i^\prime <i$, then $(gf)_{i,j}=g_{i,j}+f_{i,j}$. In particular $(f^{-1})_{i,j}=-f_{i,j}$.

\subsection{Classification of Antiholomorphic Involutions}
Any invertible antiholomorphic map can be written as the composition of the standard antiholomorphic involution coming from complex conjugation on $\bbC^{n+1}$, which we will denote by $c$, followed by an automorphism of $\bbC\bbP^n_{a_0,\dotsc,a_n}$ so let us write $\widetilde{\Aut}(\bbC\bbP^n_{a_0,\dotsc,a_n})$ for $\Aut( \bbC\bbP^n_{a_0,\dotsc,a_n})\rtimes \bbZ_2$ where $\bbZ_2$ is generated by $c$. 

By the discussion above we can write any antiholomorphic involution of $\bbC\bbP^n_{a_0,\dotsc,a_n}$ as a composition $(f,A,c)$ where $f\in H$ and $A\in \prod_{i\in I}\GL(W_i)$. Up to conjugation by elements of $\Aut(\bbC\bbP^n_{a_0,\dotsc, a_n})$ we can ignore $H$ due to the following lemma.

\begin{lem}
\label{lem:antihol}
Let $\tau=(f,A,c)\in\widetilde{\Aut}(\bbC\bbP^n_{a_0,\dotsc,a_n})$ be an antiholomorphic involution. Then $\tau$ is conjugate to $(0,A,c)$.
\end{lem}
\begin{proof}
Since $\tau$ is an involution we have $f\Ad_{A}(\bar{f})=1$. 
Let $f=f_1+\dotsb+f_m$ be the direct sum decomposition of $f$ and let $i$ be minimal such that $f_i\neq 0$. Since $i$ is minimal and the action of $\Ad_{A}$ on $H$ fixes each $H_j$ we have $(f\Ad_{A}\bar{f})_i=f_i+\Ad_{A}(\bar{f}_i)=0$.

Let $h\in H$ be such that $h_i=-\frac{1}{2}f_i$ and $h_j=0$ for $j<i$. Then $(h\tau h^{-1})_j=0$ for $j<i$ and 
\begin{align*}
(h\tau h^{-1})_i&=h_i+f_i-\Ad_{A}(\bar{h}_i)\\
&=-\frac{1}{2}f_i+f_i+\frac{1}{2}\Ad_{A}(\bar{f}_i)\\
&=\frac{1}{2}(f_i+\Ad_{A}(\bar{f}_i))\\
&=0,
\end{align*}
where we have used the linearity of the action of $\prod_{i\in I}\GL(W_i)$ on $H$ and the discussion of the group structure of $H$. Now by induction on $i$ we find that $(0,A,c)$ is in the conjugacy class of $\tau$.

%Then since $(f.A.c)^2=1=f .(A\circ c)(f).A.c(A)$ and $f_j=0$ for $j\leq i$ we have $(A\circ c)(f_i)=(f^{-1})_i=-f_i$. Let $h_i=(-\frac{1}{2}v_i,1,1)$ and $(w,A,c)=h_igh_i^{-1}$ then $w_i=-\frac{1}{2}v_i+v_i+(A\circ c)(\frac{1}{2}v_i)$. The action of $A\circ c$ on $V$ is $\bbR$-linear and hence $w_i=-\frac{1}{2}v_i+v_i+\frac{1}{2}(A\circ c)(v_i)=0$. By induction on $i$ we can find $h$ such that $hgh^{-1}=(0,A,c)$.
\end{proof}

\begin{prop}
\label{prop:antiholinv}
The number of conjugacy classes of antiholomorphic involutions of $\bbC\bbP^n_{a_0,\dotsc,a_n}$ is either $1$ or $2$. Let $k_j$, $w_j$ be defined in terms of $a_0,\dotsc,a_n$ as above. Then $\bbC\bbP^n_{a_0,\dotsc,a_n}$ admits a non-standard antiholomorphic involution if and only if $w_i k_i$ is even for each $i$.
\end{prop}
\begin{proof}

Let $\tau\in\widetilde{\Aut}(\bbC\bbP^n_{a_0,\dotsc,a_n})$ be an antiholomorphic involution. By Lemma \ref{lem:antihol} we have that $\tau$ is conjugate to $A\circ c$ so we will assume $\tau=A\circ c$. Now $A\in \prod_{i\in I}\GL(W_i)$ and for $\tau$ to be an involution we must have 
\begin{equation}
\label{eq:AcA}
A_j \overline{A}_j=\lambda^{w_j}
\end{equation}
 for some $\lambda\in\bbC^*$. Taking trace shows that $\lambda^{w_j}$ is real for each $j$ so $\lambda$ is real. By an action of the diagonal $\bbC^*$ we can ensure that $|\lambda|=1$.
Taking determinants of Eq. \eqref{eq:AcA} then implies that $\lambda^{w_jk_j}=1$.

We know that there are exactly two antiholomorphic involutions of $\bbC\bbP^{k_j}$ for $k_j$ even and one for $k_j$ odd up to conjugation and scale. For the standard antiholomorphic involution we have $A_j\overline{A}_j=1$ and for the non-standard involution we have $A_j\overline{A}_j=-1$.

For $A_j$ to be non-standard we must have $\lambda^{w_j}=-1$ so $w_j$ must be odd and $k_j$ must be even. In this case $\lambda=-1$ and since $(-1)^{w_ik_i}=1$ we must have $w_ik_i$ even for each $i$.
%any $i$ such that $k_i$ is odd we must have $w_i$ even so that $\lambda^{w_i}=1$ and for any $i$ such that $w_i$ is odd we must have that $k_j$ is even so that we can find $A_i$ non-standard.

%If $k_j$ is odd then $A_j$ must be the standard involution. 
%For $A_j$ to be the non-standard involution of $\bbC\bbP^{k_j}$ we must have that $k_j$ is even and $w_j$ is o 
%
%Suppose $k_j$ is even and $w_j$ is even then if $A_jc(A_j)=-1$ this implies that $\lambda$ is such that $\lambda^{w_j}=-1$. Since $\gcd(w_1,\dotsc,w_m)=1$ we have that there exists $w_i$ such that $w_i$ is odd and hence $\lambda^{w_i}\notin\{ 1,-1\}$. This is a contradiction so if $k_j$ is even and $w_j$ is even we must have $A_jc(A_j)=1$.
%
%Hence the only possibility for $A_jc(A_j)=-1$ is if $k_j$ is even and $w_j$ is odd. WLOG $\lambda=-1$ and so for all $i$ such that $w_i$ is odd we must have $k_i$ even so that we can find $A_i$ wuch that $A_ic(A_i)=-1$.
\end{proof}

The standard antiholomorphic involution has a fixed point locus of (real) dimension $n$ so the fixed point locus of the involution restricted to a hypersurface will never consist of isolated points. We therefore must consider only weighted projective spaces which admit non-standard involutions.

Let $\tau$ be a non-standard antiholomorphic involution and let $w_j,k_j$ be defined as before. The fixed point locus of $\tau\circ c$ is of (complex) dimension 
\[\sum_{j:w_j\in 2\bbZ}\!\!\!{k_j}-1.\]
%The intersection of a generic hypersurface with the fixed point locus of $\tau\circ c$ therefore has dimension 
%\[
%(\sum_{j:\text{ $w_j$ is even}}k_j)-2
%\].
Since we want $\tau$ to have isolated fixed points when acting on a generic hypersurface we therefore require that $\sum_{j:w_j\in 2\bbZ}k_j=2$.

For the case we are interested in, namely $\bbC\bbP^5_{a_0,\dotsc,a_n}$ the discussion above imposes conditions on the allowed sets of weights. In order for $\bbC\bbP^5_{a_0,\dotsc,a_n}$ to admit an antiholomorphic involution whose fixed locus has (real) dimension 1 we must have, without loss of generality, $a_0=a_1$ and $a_2=a_3$, all of which are odd, and $a_4$, $a_5$ both even. The action of $\tau$ on $\bbC\bbP^5_{a_0,\dotsc,a_5}$ can be given as
\begin{equation}
    \label{eq:ActionofSigma}
    \tau:[z_0,z_1,z_2,z_3,z_4,z_5]\mapsto [\overline{z}_1,-\overline{z}_0,\overline{z}_3,-\overline{z}_2,\overline{z}_4,\overline{z}_5].
\end{equation}
From now on we will assume $\tau$ is of this form.
  
  \section{Singularities}
  \label{sec:DesirableSingularities}
  Recall that we require the Calabi--Yau 4-orbifold to have singularities of the type $\frac{1}{4}(1,1,1,1)$, which are fixed by $\tau$.
  The fixed point locus of $\tau$ is contained in $S_{4,5}$ so we should find hypersurfaces with singularities of the correct type in $S_{4,5}$. 
%We therefore require that the singularities of a generic Calabi--Yau 4-fold in the weighted projective space under consideration are either of the specified type or admit crepant resolutions so that the resolved Calabi--Yau 4-orbifold has isolated singularities of the correct type.

Suppose the weights $a_0,\dotsc,a_n$ have been chosen so that a generic hypersurface of degree $d=\sum_ia_i$ is well-formed and quasismooth.
The isolated singularities of type $\frac{1}{4}(1,1,1,1)$ can 
occur in two ways, either $Y_d$ transversely intersects the singular locus $S_{4,5}$ at a generic point and $\gcd(a_4,a_5)=4$ or $Y_d$ contains a point $S_{4}$ or $S_5$ with $a_4=4$ or $a_5=4$ respectively.

For $Y_d$ to intersect a generic point of $S_{4,5}$ transversely we must have that there exist at least two monomials $z_4^{d_4}z_5^{d_5}$ of degree $d$. These singularities are of the type $\frac{1}{4}(1,1,1,1)$ if $a_k\equiv 1\mod 4$ for $k\neq 4,5$.

$Y_d$ contains the point $S_{4}$ if $a_4 \nmid d$ so that there does not exist a monomial $z_4^{d_4}$ of degree $d$. For $Y_d$ to be quasismooth we must have that $a_4 \mid d-a_j$ for some $j$ so there exists a monomial $z_4^{d_4}z_j$ of degree $d$.  As before in order for this singularity to be of the type $\frac{1}{4}(1,1,1,1)$ we must have $a_k\equiv 1\mod 4$ for $k\neq 4,j$, however $a_4$ and $a_5$ are both even so therefore we must have $j=5$.

%Suppose $Y_d\subset \bbC\bbP^5_{a_0,\dotsc,a_5}$ has $\bbC^4/\bbZ_4$ singularities described above. Without loss of generality assume that the singular points are generic points in $S_{4,5}$ or the singular point is $S_{5}$ and $a_5\mid d-a_4$. 
%In either case these singularities are isolated fixed points of an antiholomorphic involution $\tau$ iff $\tau$
%acts as
%\[
%(z_0,\dotsc,z_5)\mapsto(\overline{z}_1,-\overline{z}_0,\overline{z}_3,-\overline{z}_2,\overline{z}_4,\overline{z}_5)
%\]
%so we must have $a_0=a_1$ and $a_2=a_3$ and $a_4,a_5$ must be even.

If $\gcd(a_4,a_5)=2$ then for $Y_d$ to be quasismooth we must have that either $Y_d$ intersects $S_{4,5}$ transversely at generic points with singularities modelled on $\bbC^4/\bbZ_2$ or we have a monomial $z_4z_5$ of degree $d$. We eliminate the first possibility because  the singularity of type $\frac{1}{2}(1,1,1,1)$ does not admit a crepant resolution and we eliminate the second because $d=\sum_ia_i$. Hence $\gcd(a_4,a_5)=4$ and $Y_d$ does not contain $S_4$ or $S_{5}$.
We summarize our results in the following proposition.

\begin{prop}
    \label{prop:CorrectSings}
    Suppose the generic hypersurface of degree $d=\sum_ia_i$ in $\bbC\bbP^5_{a_0,\dotsc,a_5}$ has isolated singularities of type $\frac{1}{4}(1,1,1,1)$ and $\bbC\bbP^5_{a_0,\dotsc,a_5}$ admits an antiholomorphic involution, which fixes only these points in $Y_d$, then without loss of generality the weights $a_0,\dotsc,a_5$ satisfy
\begin{itemize}
        \item[(i)] $a_0=a_1$ and $a_2=a_3$, and
        \item[(ii)] $\gcd(a_4,a_5)=4$, and
        \item[(iii)] $a_i\equiv 1\mod 4$  for $0\leq i\leq 3$, and 
        \item[(iv)] $a_4|d$ and $a_5|d$.
\end{itemize}
\end{prop}

%To summarize, for $Y_d$ to have singularities  of the type $\frac{1}{4}(1,1,1,1)$ and for $\bbC\bbP^5_{a_0,\dotsc,a_5}$ to admit an antiholomorphic involution which fixes only these points we must have
%\begin{enumerate}
%\item $\gcd(a_4,a_5)=4$ 
%\item $a_i\equiv 1\mod 4$  for $0\leq i\leq 3$ 
%\item $a_0=a_1$ and $a_2=a_3$
%\item $a_4|d$ and $a_5|d$
%\end{enumerate}

\subsection{Resolving Undesired Singularities}
$Y_d$ may have other singularities, which we first need to resolve.
%As we have already said the existence of crepant resolutions is not an easy question.
We will use methods from \cite{dais2006existence} to determine whether a given cyclic quotient singularity of dimension 4 admits a crepant resolution.

We will assume that the reader is familiar with the basic definitions of toric geometry \cite{Fulton:1993fk}. 
Consider a cyclic quotient singularity of the type $\frac{1}{m}(a_1,a_2,a_3,a_4)$. We can describe this as an affine toric variety.
Let $N=\bbZ^4+\bbZ\cdot\frac{1}{m}(a_1,a_2,a_3,a_4)$ be a lattice and $\sigma\subset N_\bbQ=N\otimes_\bbZ\bbQ$ the cone spanned by the unit vectors $e_1=(1,0,0,0),\dotsc,e_4=(0,0,0,1)$. The affine toric variety associated to the cone $\sigma$ is isomorphic to the cyclic quotient singularity of type $\frac{1}{m}(a_1,a_2,a_3,a_4)$. 

%Let $n_i=\frac{1}{21}(i\bmod21,i\bmod21,7i\bmod21,12i\bmod21)\in N$

The set of elements of age $i$, $\sigma_i$, is defined to be the convex hull in $N$ of the elements $\{ie_1,ie_2,ie_4,ie_4\}\in N$. 
The following theorem, from \cite[Th. 6.1]{dais2006existence}, gives a necessary condition for the cyclic quotient singularity to admit a crepant resolution.

\begin{thm}
\label{thm:neccrepres}
Let $\bbC^n/G$ be a quotient singularity, where $G\subset \SL(n,\bbC)$ is a finite abelian group. If $\bbC^n/G$ admits a crepant resolution, then the set of elements of age 1, $\sigma_1$, is a minimal generating set for $\sigma$ over $\bbZ$. 
\end{thm}

Theorem \ref{thm:neccrepres} gives us a necessary condition for a given cyclic quotient singularity to admit a crepant resolution. 
This condition is sufficient for all singularities of codimension $4$ where the cyclic group has order less than 39 and is sufficient in all but 10 cases for quotient singularities of codimension $4$ with cyclic group of order less than 100 \cite{dais2006existence}.

\begin{exmp}
Consider the isolated cyclic quotient singularity of type $\frac{1}{2}(1,1,1,1)$. 
The elements of age 1 are $e_1,\dotsc,e_4$.
The element $\frac{1}{2}(1,1,1,1)\in \sigma$ cannot be written as a sum of $e_1,\dotsc,e_4$ with integer coefficients.
Therefore the elements of age 1 do not form a generating set for $\sigma$ over $\bbZ$ and hence the singularity of type $\frac{1}{2}(1,1,1,1)$ does not admit a crepant resolution.
\end{exmp}

\begin{exmp}
    The generic hypersurface, $Y_{84}$, of degree 84 in $\bbC\bbP^5_{1,1,21,21,12,28}$ is a well-formed quasismooth Calabi--Yau hypersurface. The singularities of $Y_{84}$ consist of the curves $Y_{84}\cap S_{2,3,4}$ and $Y_{84}\cap S_{2,3,5}$, which intersect in the 4 points $\{p_1,\dotsc,p_4\}=Y_{84}\cap S_{2,3}$. The singularities of $Y_{84}$ at each of the $p_i$ are of type $\frac{1}{21}(1,1,7,12)$.

Let $N=\bbZ^4+\bbZ\cdot\frac{1}{21}(1,1,7,12)$ and $\sigma\subset N_\bbQ$ the cone spanned by $e_1=(1,0,0,0),\dotsc,e_4=(0,0,0,1)$. 
%Let $n_i=\frac{1}{21}(i\bmod21,i\bmod21,7i\bmod21,12i\bmod21)\in N$
The elements of age 1, which are listed in Table \ref{tab:exampleage1}, are a minimal generating set for $\sigma$ and hence the singularity $\bbC^4/\bbZ_{21}$ admits a crepant resolution.
\begin{table}[htbp]
\begin{tabular}{r r r r}
$(1,0,0,0)$, & $(0,1,0,0)$, & $(0,0,1,0)$, & $(0,0,0,1)$, \\[3pt]
$\frac{1}{21}(1,1,7,12)$, &
$\frac{1}{21}(2,2,14,3)$,& 
$\frac{1}{21}(3,3,0,15)$, &
$\frac{1}{21}(4,4,7,6)$, \\[3pt]
$\frac{1}{21}(6,6,0,9)$, &
$\frac{1}{21}(7,7,7,0)$, &
$\frac{1}{21}(9,9,0,3)$, \\[3pt]
\end{tabular}
%\begin{tabular}{r r r r}
%$(1,0,0,0)$ & $(0,1,0,0)$ & $(0,0,1,0)$ & $(0,0,0,1)$ \\[3pt]
%$\frac{1}{21}(1,1,7,12)$ &
%$\frac{1}{21}(2,2,14,3) $& 
%$\frac{1}{7}(1,1,0,5)$ &
%$\frac{1}{21}(4,4,7,6)$ \\[3pt]
%$\frac{1}{7}(2,2,0,3)$ &
%$\frac{1}{3}(1,1,1,0)$ &
%$\frac{1}{7}(3,3,0,1)$ \\[3pt]
%\end{tabular}
\bigskip
\caption{\label{tab:exampleage1}Elements of age 1 in the cone $\sigma$, which defines the cyclic quotient singularity of type $\frac{1}{21}(1,1,7,12)$.}
\end{table}
\end{exmp}

We are now in a position to determine whether a particular weighted projective space contains a suitable Calabi--Yau 4-orbifold as a well-formed quasismooth hypersurface.
\begin{prop}
  \label{prop:AdmissableWeights}
The weights $a_0,\dotsc,a_5$ such that 
\begin{itemize}
\item[(i)] The generic hypersurface of degree $d=\sum_ia_i$ in $\bbC\bbP^5_{a_0,\dotsc,a_5}$ is well-formed and quasismooth; 
\item[(ii)] $Y_d$ has isolated singularities of the type $\frac{1}{4}(1,1,1,1)$;
\item[(iii)] $\bbC\bbP^5_{a_0,\dotsc,a_5}$ admits an antiholomorphic involution whose fixed point locus intersects $Y_d$ at the isolated singularities of type $\frac{1}{4}(1,1,1,1)$;
\item[(iv)] Any other singularities of $Y_d$ admit crepant resolutions;
\end{itemize}  
are listed in Table $\ref{tab:AdmissableWeights}$. 
\end{prop}

\begin{proof}
  Lynker et al. \cite{lynker1999landau} determined the complete set of weighted projective spaces of dimension $5$ such that the generic hypersurface of degree $d=\sum_ia_i$ is quasismooth. 
  The list of weights can be found at \url{http://thp.uni-bonn.de/Supplements/cy.html}.

  Propositions \ref{prop:GenericHypersurfaceWellFormed} and \ref{prop:CorrectSings} translate conditions (i)--(iii) into numerical conditions on the weights $a_0,\dotsc,a_5$.
  %We use Proposition \ref{prop:GenericHypersurfaceWellFormed} to determine whether the generic hypersurface of degree $d$ is well-formed.
  %We may use Proposition \ref{prop:CorrectSings} to determine whether the weights satisfy conditions (ii) and (iii).
  %By the discussion following Proposition \ref{prop:antiholinv} and that is Section \ref{sec:DesirableSingularities} we know that for (ii) and (iii) to hold we must have $a_0=a_1$, $a_2=a_3$, $a_i\equiv 1\mod 4$ for $0\leq i\leq 3$ and $\gcd(a_4,a_5)=4$.
%
  We use a computer programme to search the list of 1,100,055 sets of weights given by Lynker to get a list of $18$ sets of weights, for which conditions (i)--(iii) apply.

   Then we use Theorem \ref{thm:neccrepres} to test whether any undesired singularities of the generic hypersurface admit crepant resolutions. 
   This test eliminates the weights which are not listed in Table \ref{tab:AdmissableWeights}. 
\end{proof}
\begin{table}[htbp]
\begin{tabular}{|r r r r r r| r r r r r r|}
    \hline
    \multicolumn{12}{|c|}{$\{a_0,\dotsc,a_5\}$} \\
    \hline
    1& 1& 1& 1& \phantom{2}4& 4 & 
\phantom{2}1& \phantom{2}1& 1& 1& 4& 8 \\
1& 1& 1& 1& 8& 12 & 
1& 1& 5& 5& 8& 20 \\
\hline
1& 1& 9& 9& 4& 4 &
5& 5& 13& 13& 4& 4 \\
1& 1& 13& 13& 4& 8 &
1& 1& 21& 21& 4& 16 \\
5& 5& 25& 25& 4& 16 &
1& 1& 21& 21& 12& 28 \\
1& 1& 37& 37& 8& 28 &
1& 1& 53& 53& 20& 32 \\
21& 21& 49& 49& 4& 24 &
1& 1& 69& 69& 16& 52 \\
\hline
\end{tabular}
%\begin{tabular}{| r r r r r r | }
%\hline
%\multicolumn{6}{|c|}{$\{a_0,\dotsc,a_5\}$} \\
%\hline
%1&1&1&1&4&4\\
%1&1&1&1&4&8\\
%1&1&1&1&8&12\\
%1&1&5&5&8&20\\
%\hline
%1&1&9&9&4&4\\
%5&5&13&13&4&4\\
%1&1&13&13&4&8\\
%1&1&21&21&4&16\\
%5&5&25&25&4&16\\
%1&1&21&21&12&28\\
%1&1&37&37&8&28\\
%1&1&53&53&20&32 \\
%21&21&49&49&4&24\\
%1&1&69&69&16&52 \\
%\hline
%\end{tabular}
\bigskip
\caption{\label{tab:AdmissableWeights}The admissable weights of the ambient weighted projective spaces of Calabi--Yau 4-orbifolds. The weighted projective spaces with weights listed in the first two rows appear as ambient spaces for Calabi--Yau 4-orbifolds in \cite{Joyce:1996st}.}
\end{table}

\section{Determining Betti Numbers of $M$}

Let us now suppose that $Y$ is a Calabi--Yau 4-orbifold contained in one of the weighted projective spaces we have determined in Proposition \ref{prop:AdmissableWeights}.
In general $Y$ will have singularities, which we first need to resolve.
We will denote the resolution of $Y$ by $\hat{Y}$ and let us assume that $\tau$ lifts to $\hat{Y}$ so that $\hat{Y}$ satisfies Condition \ref{cond:Y}.
%We will assume that we have resolved any unwanted singularities of $Y$ so that $Y$ satisfies Condition \ref{cond:Y}. 

For the moment let us assume that we can determine the Hodge numbers of $\hat{Y}$. We can determine the Betti numbers of $Z=\hat{Y}/\langle \tau\rangle$ and the resulting $\mathrm{Spin}(7)$-manifold $M$ as follows.

\begin{prop}
  \label{prop:BettiNumbersofMintermsofZ}
  Let $\hat{Y}$, $Z$, $M$ be as above.
  Suppose $Z$ has $k$ singularities modelled on $\bbR^8/G$ as in Section \ref{sec:ReviewofConstruction}.
  Then the Betti numbers of $M$ are
\begin{align*}
b^2(M)&=b^2(Z), & b^4_+(M)&=\frac{1}{2}(h^{2,2}(\hat{Y})+k)-b^2(Z)+1,\\
b^3(M)&=\frac{1}{2}b^3(\hat{Y}), & b^4_-(M)&=h^{3,1}(\hat{Y})+b^2(\hat{Y})-b^2(Z)+k-1. 
\end{align*}
\end{prop}
\begin{proof}
  
  Let $h^{p,p}_\tau(\hat{Y})$ be the dimension of the $\tau$-invariant part of $H^{p,p}(\hat{Y})$.
  Noting that in all of the cases we will discuss we have $h^{2,0}(\hat{Y})=h^{3,0}(\hat{Y})=0$ and $h^{4,0}(\hat{Y})=1$ since $\hat{Y}$ has $\Hol(\hat{Y})=SU(4)$,
  the Betti numbers of $Z$ can then be expressed as
%  \begin{align*}
%   b^2(Z)&=h^{2,0}(\hat{Y})+h^{1,1}_\tau(\hat{Y}) & b^3(Z)&=h^{3,0}(\hat{Y})+h^{2,1}(\hat{Y}) \\
%    b^4(Z)&=h^{4,0}(\hat{Y})+h^{3,1}(\hat{Y})+h^{2,2}_\tau(\hat{Y})
%  \end{align*}
  \begin{align*}
 b^2(Z)&=h^{1,1}_\tau(\hat{Y}), & b^4_+(Z)&=h^{2,2}_\tau(\hat{Y})-h^{1,1}(\hat{Y})+h^{1,1}_\tau(\hat{Y})+2, \\
    b^3(Z)&=h^{2,1}(\hat{Y}), & b^4_-(Z)&=h^{3,1}(\hat{Y})+h^{1,1}(\hat{Y})-h^{1,1}_\tau(\hat{Y})-1.
  \end{align*}

  Applying the Lefschetz fixed point theorem we find that
 % $k,h^{1,1},h^{1,1}_\tau,h^{2,2},h^{2,2}_\tau$ so we choose to eliminate $h^{2,2}_\tau$. 
 \begin{equation*}
   k=2+4h^{1,1}_\tau(\hat{Y})-2h^{1,1}(\hat{Y})+2h^{2,2}_\tau(\hat{Y})-h^{2,2}(\hat{Y}), 
 \end{equation*}
 which we use to eliminate $h^{2,2}_\tau(\hat{Y})$ from the expressions for the Betti numbers of $Z$. 
  
 The ALE $\mathrm{Spin(7)}$ manifolds that we use to resolve the quotient singularities of $Z$ have Betti numbers $b^1=b^2=b^3=b^4_+=0$ and $b^4_-=1$ hence the Betti numbers of $M$ satisfy
  \begin{equation*}
    \begin{split}
    b^j(M)&=b^j(Z) \text{ for $j=1$, $2$, $3$} \\
    b^4_+(M)&= b^4_+(Z) \text{ and } b^4_-(M)=b^4_-(Z)+k.
  \end{split}
\end{equation*}

Combining these facts gives us the result.
\end{proof}

From Proposition \ref{prop:BettiNumbersofMintermsofZ} we see that to determine the Betti numbers of $M$ it suffices to know the Hodge numbers of the orbifold $\hat{Y}$ and to understand how $\tau$ acts on $H^{1,1}(\hat{Y})$. 
We will use techniques from toric geometry both to determine the Hodge numbers and to understand the action of $\tau$. 

The rest of this section will almost entirely be material from \cite{batyrev1994dual}, and we direct the reader to this paper for more details.

%
%Suppose we now have a Calabi--Yau 4-fold, $\hat{Y}$, which satisfies Condition \ref{cond:\hat{Y}} and as before let $Z=\hat{Y}/\langle \tau\rangle$. We wish to find the Betti numbers of the resulting $\mathrm{Spin}(7)$-manifold $M$. 
%Suppose $\tau$ fixes $k$ points of $\hat{Y}$ then from \cite[\S 15.3.4]{joyce2000compact} 
%\begin{align*}
%b^2(M)&=b^2(Z) & b^4_+(M)&=\frac{1}{2}(h^{2,2}(\hat{Y})+k)-b^2(Z)+1\\
%b^3(M)&=\frac{1}{2}b^3(\hat{Y}) & b^4_-(M)&=h^{3,1}(\hat{Y})+b^2(\hat{Y})-b^2(Z)+k-1 
%\end{align*}
%
%We see that to find the Betti numbers of $M$ it suffices to know the Hodge numbers of $\hat{Y}$ and $b^2(Z)$, which is the rank of the $\tau$-invariant part of $H^2(\hat{Y},\bbC)$.

\subsection{Lattice Polytopes}
We will now change our viewpoint from weighted projective spaces to toric varieties associated to reflexive polytopes.
With this change we will find the Hodge numbers of the resolved hypersurface $\hat{Y}$, show that the antiholomorphic involution $\tau$ lifts to $\hat{Y}$, and determine the dimension of $H^2_\tau(\hat{Y})$.

Batyrev and Cox \cite{batyrev1994dual,batyrev1996strong} have determined the Hodge numbers of the crepant resolutions of Calabi--Yau hypersurfaces in toric varieties associated to reflexive polytopes.
We will give the definition of reflexive polytope and show how to associate a toric variety to such a polytope.

In this subsection $N$ will denote a lattice and $M=\Hom(N,\bbZ)$ its dual.
We will denote a fan by $\Sigma$ and  a rational strongly convex cone by $\sigma$.
%We will now describe a construction of a simplicial toric variety which is determined by a polytope in $M$.
Let $\Delta\subset M$ be an $n$-dimensional lattice polytope, i.e. a polytope with vertices in $M$, and suppose $\Delta$ contains the origin. We associate a toric variety to $\Delta$ by taking cones over the maximal faces of $\Delta$ as described in the following proposition.
\begin{prop}
    For every $k$-dimensional face $\Theta\subset \Delta$ let $\check{\sigma}(\Theta)\subset M_\bbQ=M\otimes_\bbZ\bbQ$ be the cone over $\Theta$, $\check{\sigma}(\Theta)=\{\lambda x\in M_\bbQ\text{ : $x\in \Theta$ and $\lambda \in \bbQ$}\}$, and let $\sigma(\Theta)\subset N_\bbQ$ be the $(n-k)$-dimensional dual cone. Then the collection of dual cones $\Sigma(\Delta)=\{\sigma(\Theta)\text{ : $\Theta\subset\Delta$}\}$ is a fan and hence determines a toric variety $P_\Delta$.
\end{prop}

\begin{exmp}
We recall how weighted projective space can be constructed as a toric variety.
%As an example of a toric variety we will describe weighted projective space $\bbC\bbP^n_{a_0,\dotsc,a_n}$.
Let $a_0,\dotsc,a_n$ be positive integers with $\gcd(a_0,\dotsc,a_n)=1$.
Let $\overline{N}$ be generated by $e_0,\dotsc,e_{n}$ and let $N=\overline{N}/\bbZ\cdot(a_0e_0+\dotsb+a_ne_n)$.
$N$ is a lattice since $\gcd(a_0,\dotsc,a_n)=1$.
$\bbC\bbP^n_{a_0,\dotsc,a_n}$ is the toric variety associated to the the fan whose $n$ dimensional cones are given by $\linspan(e_0,\dotsc,e_{i-1},e_{i+1},\dotsc,e_n)$ for $i=0,\dotsc,n$ in $N_\bbQ$.
\end{exmp}

\begin{exmp} 

    Let $\overline{N}$ be generated by $e_0,\dotsc,e_n$ and let $\overline{M}$ be the dual lattice. Suppose $x\in \bbZ_{>0}e_0+\dotsb+\bbZ_{>0}e_n$ is a primitive element in $\overline{N}$, i.e. $x=a_0e_0+\dotsb+a_ne_n$ where $\gcd(a_0,\dotsc,a_n)=1$ and $a_i>0$ for each $i$. 
Then the set $\Delta=\{y\in \overline{M}\text{ : $\langle y,x\rangle=0$ and $\langle y,e_i\rangle\geq-1$}\}$ is an $n$-dimensional lattice polytope in the lattice $M=\{y\in\overline{M}\text{ : $\langle y,x\rangle=0$}\}$.

The dual lattice $N=\Hom(M,\bbZ)$ can be identified with the quotient $\overline{N}/\bbZ\cdot(a_0e_0+\dotsb+a_ne_n)$. If $\gcd(a_0,\dotsc,\widehat{a_i},\dotsc,a_n)=1$ for each $i$ so that $\bbC\bbP^n_{a_0,\dotsc,a_n}$ is well-formed, then the fan $\Sigma(\Delta)$ is a refinement of the fan of $\bbC\bbP^n_{a_0,\dotsc,a_n}$ so the toric variety $P_\Delta$ is a partial resolution of $\bbC\bbP^n_{a_0,\dotsc,a_n}$.

\end{exmp}

The previous example shows that to any weighted projective space we can associate a lattice polytope, and after applying the toric construction to the polytope we get a toric variety, which is a partial resolution of the original weighted projective space. 

A polytope $\Delta$ determines not only a toric variety $P_\Delta$ but also a choice of ample invertible sheaf $\mathcal{O}_\Delta(1)$ on $P_\Delta$ by \cite[Prop. 2.1.5]{batyrev1994dual}. 

\begin{defn}
Let $M$ be a lattice and $\Delta\subset M$ a lattice polytope of the same dimension as $M$ containing the origin.
We define the \emph{dual polytope} $\Delta^*\subset N_\bbQ$ as the set
\[
\Delta^*=\{x\in N_\bbQ\text{ : $\langle y,x\rangle\geq -1$ for all $y\in\Delta$}\}
\]
We say a lattice polytope $\Delta$ is \emph{reflexive} if $\Delta^*$ is also a lattice polytope, i.e. if the vertices of $\Delta^*$ lie in $N$ and not just $N_\bbQ$.
\end{defn}

The relevance of reflexive polytopes to Calabi--Yau orbifolds is described in the following result from \cite[Th. 4.1.9]{batyrev1994dual}.

\begin{thm}
  Let $\Delta$ be an integral polytope and $P_\Delta$ the corresponding projective toric variety. The following conditions are equivalent.
  \begin{enumerate}
    \item the ample invertible sheaf $\mathcal{O}_\Delta(1)$ on $P_\Delta$ is anticanonical;
    \item $\Delta$ is reflexive.
  \end{enumerate}
\end{thm}

Now suppose $\Delta$ is reflexive. Then a generic section of the sheaf $\mathcal{O}_\Delta(1)$ determines a quasismooth Calabi--Yau hypersurface, by an application of the adjunction formula.

The lattice polytopes associated to weighted projective spaces as we have described above are not always reflexive.
For example the  weights 1, 1, 1, 1, 1, 2 do not determine a reflexive polytope.
Fortunately the lattice polytopes associated to the weights described in Proposition \ref{prop:AdmissableWeights} are all reflexive.

The Hodge numbers of the resolved hypersurface $\hat{Y}$ are given in terms of combinatorial properties of the lattice polytope $\Delta$ defined by the weights $a_0,\dotsc,a_5$. We will not give the formulae here but direct the reader to \cite{batyrev1994dual,batyrev1994hodge}.

\subsection{$\tau$-Equivariant resolutions of singularities}
If $Y\subset P_\Delta$ is a Calabi--Yau hypersurface then a crepant resolution of the singularities of $P_\Delta$ will induce a crepant resolution of the singularities of $Y$.
We will resolve the singularities of $P_\Delta$ and in the process desingularize all of the Calabi--Yau hypersurfaces.

%Let $\Delta$ be a reflexive polytope and $P_\Delta$ the associated toric variety. 
%Recall that the fan of $P_\Delta$ is defined to be the cone over the faces of the dual polytope $\Delta^*$.
A subdivision of the polytope $\Delta^*$ will determine a refinement of the fan of $P_\Delta$ and hence a partial resolution of $P_\Delta$.
A subdivision of $\Delta^*$ is called a \emph{(maximal) triangulation} if every lattice point in $\Delta^*$ is the vertex of some simplex in the subdivision.
A triangulation of $\Delta^*$ will determine a maximal partial crepant resolution of $P_\Delta$ and hence $Y$.
A triangulation is called \emph{projective} if the associated resolution is.
The existence of projective triangulations is guaranteed by \cite[Prop. 4]{Gelfand:1989fk}
 and the resulting resolution will, very importantly, be crepant by \cite[Th. 2.2.24]{batyrev1994dual}.

%Singularities of codimension less than $4$ are guaranteed to be resolved by this process by \cite[Cor. 3.1.7]{batyrev1994dual}, which is why we only checked for the existence of crepant resolutions of singularities of codimension 4.

However we require that $\tau$ lifts to the resolution of $P_\Delta$, which we will denote by $\hat{P}_\Delta$.
The antiholomorphic involution $\tau$ can be decomposed into three parts
\begin{equation*}
    \tau=t\cdot \tau_m \cdot c,
\end{equation*}
where $c$ denotes the standard antiholomorphic involution, $\tau_m$ is an element of the torus in $P_\Delta$, and $t$ is a morphism of $P_\Delta$ induced by an involution of the lattice $N$, which fixes the polytope $\Delta^*$ (and $\Delta$).

The involution $\tau$ will lift to the resolution if the triangulation of $\Delta^*$ is invariant under the action of $t$, by which we mean that $t$ sends a simplex in the triangulation to another simplex in the triangulation.
Unfortunately we cannot always find such an invariant triangulation as the following example shows.

%
%A function $\phi:\Delta^*\setminus\{0\}\rightarrow \bbQ_{> 0}$ extends uniquely to a linear function on $\Delta^*_\bbQ$, which we denote by $\hat{\phi}$, such that $\hat{\phi}(x)=\phi(x)$ for $x\in\Delta^*$. The linear function $\hat{\phi}$ determines a subdivision of the polytope $\Delta^*$ by taking the convex hull of the graph of $\hat{\phi}$ in $N_\bbQ\times \bbQ$ and projecting the faces down to $N_\bbQ$.
%The set of functions which determine the same triangulation form a cone in $\bbQ^{{\Delta^*\setminus\{0\}}}$ and a triangulation is called \emph{projective} if this cone has a non-empty interior.
%
%A subdivision of $\Delta^*$ is called a \emph{() triangulation} if every lattice point in $\Delta^*$ is the vertex of some simplex in the subdivision.
%If $G$ is a finite group of lattice isomorphisms, which are symmetries of $\Delta^*$, we say a triangulation is $G$-invariant if a simplex in the triangulation is sent to another simplex in the triangulation under the action of $G$.
%Any polytope admits a projective triangulation, however not all triangulations of $\Delta^*$ are invariant under $G$ and in fact there may not be any projective triangulations which are $G$-invariant as the following example shows.
%

\begin{exmp}
    Let $e_0=(0,0,-2)$, $e_1=(1,1,1)$, $e_2=(1,-1,1)$, $e_3=(-1,-1,1)$, and $e_4=(-1,1,1)$ be points in $\bbR^3$. 
    Let $N$ be the lattice generated by $e_1,e_2,e_3$ and let $\Delta^*$ denote the reflexive polytope given by the convex hull of the set $\{e_0,\dotsc,e_4\}$.
    The polytope $\Delta^*$ can be pictured as a cone over a square with vertices $e_1,e_2,e_3,e_4$.
    %Note that $e_4=e_1-e_2+e_3$ and $e_0=-e_1-e_3$ in the lattice $N$.
    Let $t$ be the lattice isomorphism defined by $t(e_1)=e_2$, $t(e_2)=e_1$ and $t(e_3)=e_4$, which is reflection in a plane in $\bbR^3$.

    There are two triangulations of the polytope $\Delta^*$.
    One contains the edge joining $e_1$ and $e_3$ and the other contains the edge joining $e_2$ and $e_4$. 
    Neither of these triangulations is invariant under the action of $t$ and we see that the map induced by $t$ on the toric variety does not lift to any resolution.
    It is interesting to note that the polytope is not simplicial and hence the associated toric variety is not an orbifold.
    We have not found an example of a simplicial reflexive polytope with an involution $t$, which does not admit a $t$-invariant projective triangulation.

\end{exmp}

%
%We can desingularize $P_\Delta$, and in doing so desingularize Calabi--Yau hypersurfaces in $P_\Delta$, by finding a triangulation of the polytope $\Delta$, which will determine a refinement of the fan.
%A triangulation of the polytope will determine a crepant resolution of singularities of Calabi--Yau hypersurfaces in $P_\Delta$. 
%
%A maximal triangulation of $\Delta$ determines a maximal crepant resolution of a generic Calabi--Yau hypersurface in $P_\Delta$.
%Any lattice polytope admits a maximal (coherent/projective) triangulation, however we require more than simply a triangulation of the polytope. 
%We require that the antiholomorphic involution lifts to the resolution of the toric variety.
%The antiholomorphic involution can be decomposed as a composition of 
%\[
%    \tau = \tau_N\circ\tau_T\tau c
%\]
%where $c$ is the standard complex conjugation acting on any toric variety, $\tau_T$ is an action of the torus, and $\tau_N$ is an automorphism of the lattice which fixes the fan.
%Complex conjugation $c$ and the torus action naturally lift to the resolution so we only need worry about the lattice part of the involution.
%We can show that in each of the cases we are interested in the involution lifts to any maximal triangulation of the polytope.
%Note that this will not be true in general (????).
%

%The lattice involution $t$ acts on the polytope $\Delta^*$ and we must ensure that we can find a triangulation which is invariant under this involution. 

Projective triangulations of the lattice polytope, $\Delta^*$, are in one-to-one correspondence with the faces of an associated polytope, known as the secondary polytope \cite[Ch. 7]{Gelfand:1994fk}.
For $\Delta^*$ to admit a $t$-invariant triangulation, we require that $t$ must fix a face of the secondary polytope, or equivalently the fixed point set of $t$ intersects the interior of a face of the secondary polytope.
The secondary polytope sits inside the vector space $A_{n-1}(\hat{P}_\Delta)\otimes \bbQ$, where $\hat{P}_\Delta$ is any maximal partial crepant resolution of $P_\Delta$.
Recall that $A_{n-1}(\hat{P}_\Delta)$ is determined by an exact sequence \cite{Fulton:1993fk}
\begin{equation*}
    \label{eq:DefOfChowGroup}
    0\longrightarrow M\longrightarrow \sum_{\rho\in\Delta^*\setminus{\{0\}}} \bbZ\cdot e_\rho\longrightarrow A_{n-1}(\hat{P}_\Delta)\longrightarrow 0,
\end{equation*}
where $m\in M\mapsto \sum_{\rho\in\Delta^*\setminus{\{0\}}}\langle m,\rho\rangle e_\rho$.
The lattice isomorphism $t$ acts on $M$ and $\Delta^*$ and hence on $A_{n-1}(\hat{P}_\Delta)$.
If we tensor this sequence with $\bbQ$ we can check that for all of the cases we are interested in, $t$ in fact fixes the whole secondary polytope.
This means that $\tau$ will lift to any maximal crepant resolution of $P_\Delta$.

%%If a function $\phi:\Delta^*\setminus{\{0\}}\rightarrow \bbQ_{>0}$ determines a triangulation which is invariant under the involution then by averaging we can assume that $\phi$ itself is invariant under the involution.
%%Note that if two functions $\phi_1,\phi_2\in \bbQ_{>0}^{\Delta^*\setminus{\{0\}}}$ differ by a function of the form $\rho\mapsto \langle m,\rho\rangle$ for some $m\in M_\bbQ$ then they determine the same triangulation.
%%By counting eigenvalues we find that in each of our cases the quotient $\bbQ^{\Delta^*\setminus{\{0\}}}/M$ is invariant under the involution and hence in these cases all projective maximal triangulations are invariant under the involution.

%
%Define $A_{\Delta^*}$ by the sequence
%\begin{equation}
%    0\rightarrow M\otimes\bbQ\rightarrow \sum_{\rho\in\Delta^*\setminus\{0\}}\bbQ. e_\rho\rightarrow A_{\Delta^*}\rightarrow 0
%\end{equation}
%and $A^+_{\Delta^*}$ as the image of the $\sum_{\rho\in\Delta^*\setminus\{0\}}\bbQ_+. e_\rho$
%
%from because a generic point in $A_{n-1}P_\Delta\otimes\bbQ$ determines a projective triangulation, see [ref, ??], and using the exact sequence

\subsection{Hodge Numbers of $Y$}

The Hodge numbers of the resolved Calabi--Yau hypersurface, $\hat{Y}$, are given by formulae presented in \cite{batyrev1994dual,batyrev1996strong}.
To understand how $\tau$ acts on $H^{1,1}(\hat{Y})$ we will describe in more detail how $h^{1,1}(\hat{Y})$ is calculated. 

The basic idea is that $h^{1,1}(\hat{Y})$ counts the components of the intersection of the resolution with the union of all irreducible toric divisors in the desingularization of $\bbC\bbP^n_{a_0,\dotsc,a_n}$.
If the resolution intersects a divisor with dimension $>0$, then it is irreducible, so the action of $\tau$ on this component can be determined by whether $\tau$ fixes the toric divisor or swaps it with another.

We denote the torus in $\hat{P}_\Delta$ by $T$ and let $X$ denote the intersection of $\hat{Y}$ with the union of all irreducible $T$-invariant divisors in $\hat{P}_\Delta$.
We have a short exact sequence of cohomology groups 
\begin{equation*}
    0\longrightarrow H^6_c(\hat{Y})\longrightarrow H^6_c(X)\longrightarrow H^7_c(\hat{Y}\setminus X)\longrightarrow 0,
\end{equation*} 
which is natural under $\tau$.
By Poincar\'e duality we have that $b^2(\hat{Y})=\dim(H^6_c(\hat{Y}))$ and hence $b^2_\tau(\hat{Y})$ is the difference of the dimension of the $\tau$-invariant parts of $H^6_c(X)$ and $H^7_c(\hat{Y}\setminus X)$.
Using a Lefchetz-type theorem \cite{Danilov:1986uq} for affine hypersurfaces in algebraic tori we have that $H^7_c(\hat{Y}\setminus X)\simeq H^9_c(T)$ is an isomorphism.
Given the description of $\tau$ as in Eq. \eqref{eq:ActionofSigma}, it is easy to check that the dimension of the $\tau$-invariant part of $H^9_c(T)$ is 3.

The irreducible $T$-invariant divisors in $\hat{P}_\Delta$ are indexed by the set $\Delta^*\setminus{\{0\}}$.
Let $\rho\in \Delta^*\setminus{\{0\}}$ and $D_\rho$ the corresponding $T$-invariant divisor.
If $\rho$ is contained in the interior of a face of codimension $\geq 3$, then the intersection $D_\rho\cap \hat{Y}$ is irreducible, while if $\rho$ is contained in the interior of a face of codimension $2$, then $D_\rho$ intersects $\hat{Y}$ in isolated points, the number of which is determined by the polytope.

If $D_\rho\cap \hat{Y}$ is irreducible, then the action of $\tau$ is determined by whether $\tau$ fixes $\rho\in \Delta^*\setminus{\{0\}}$ or not.
If $\tau$ fixes $\rho$, then $\rho$ does not contribute to $b^2_\tau(\hat{Y})$ and a pair $\rho_1,\rho_2$ of $T$-invariant divisors, which are swapped by $\tau$, contribute 1 to $b^2_\tau(\hat{Y})$.

If $D_\rho$ intersects $\hat{Y}$ in $d$ points, then for a generic $Y$, $\tau$ will swap $d/2$ pair of points if $d$ is even and $(d-1)/2$ pairs of points if $d$ is odd.
However we can choose $Y$ so that $\tau$ swaps $k$ pairs of points where $0\leq k\leq d/2$, in which case this contributes $k$ to $b^2_\tau(\hat{Y})$.
In this way we see that we can get $\mathrm{Spin(7)}$-manifolds with different topological invariants arising from the same family of Calabi--Yau 4-orbifolds.

%If the resolution intersects a divisor with dimension $0$ then a generic Calabi--Yau hypersurface will intersect it in a (combinatorial) number of ways. The points of intersection can be fixed by $\tau$ or swapped in pairs but this will depend on the Calabi--Yau in the family. In this way we can get different Betti numbers of $\mathrm{Spin}(7)$-manifolds coming from Calabi--Yau manifolds in the same family.

We have used the software PALP \cite{kreuzer2004palp} to find the toric divisors and to determine the toric divisors fixed by $\tau$.

\begin{exmp}
    Consider the reflexive polytope with weights $1,1,9,9,4,4$ $\Delta\subset M$. Let $N=M^*$ and $\Delta^*$ the dual polytope of $\Delta$.
    The points of $\Delta^*\setminus\{0\}$ correspond to toric divisors in $P_\Delta$. 
   There are exactly 11 of these, which are listed in Table \ref{tab:ToricDivisorsExample} with respect to a particular basis of $N$. 
   The antiholomorphic involution swaps the elements in the first column in pairs and leaves the other 7 invariant.  
   \begin{table}[htbp]
       \begin{tabular}{l l l l}
           $(\phantom{-}1, \phantom{-}0, \phantom{-}0, \phantom{-}0, \phantom{-}0)$, & $(0, \phantom{-}0, \phantom{-}0, \phantom{-}1, \phantom{-}0)$, & $(0, -5, -5, -2, -2)$, \\
           $(\phantom{-}0, \phantom{-}1, \phantom{-}0, \phantom{-}0, \phantom{-}0)$, & $(0, \phantom{-}0, \phantom{-}0, \phantom{-}0, \phantom{-}1)$, & $(0, -3, -3, -1, -1)$, \\
           $(\phantom{-}0, \phantom{-}0, \phantom{-}1, \phantom{-}0, \phantom{-}0)$, & $(0, -7, -7, -3, -3)$, & $(0, -2, -2, -1, -1)$, \\
           $(-1, -9, -9, -4, -4)$, & $(0, -1, -1, \phantom{-}0, \phantom{-}0)$, &
   \end{tabular}
   \bigskip
   \caption{\label{tab:ToricDivisorsExample} $T$-invariant divisors in a maximal partial crepant resolution of the reflexive polytope with weights $1,1,9,9,4,4$.}
   \end{table}
   
   There is one divisor contained in an interior of a face with codimension 2, which corresponds to the singularities of our Calabi--Yau hypersurface of type $\frac{1}{4}(1,1,1,1)$.
   A generic hypersurface intersects the divisor in 7 points.
   We can choose the hypersurface so that $\tau$ swaps $k$ pairs of points where $0\leq k\leq 3$ and fixes the remaining singular points.
   After taking the quotient by $\tau$ and resolving the quotient singularities in the usual way, we can find $\mathrm{Spin}(7)$ manifolds with $0\leq b^2(M)\leq 3$.
   
\end{exmp}

\section{Results and further work}

Table \ref{tab:results} lists the examples of $\mathrm{Spin}(7)$-manifolds constructed from well-formed quasismooth hypersurfaces in weighted projective spaces. We include the examples already given in \cite[Ch. 15]{joyce2000compact} for the sake of completeness, which appear as the first four rows.

\begin{table}[tbp]
\begin{tabular}{| r r r r r r | c|  c| c| c| }
\hline
\multicolumn{6}{|c|}{$\{a_0,\dotsc,a_5\}$} & $b^2$ & $b^3$ & $b^4_+$ & $b^4_-$ \\
\hline
1&1&1&1&4&4&$0\leq k\leq 1$ & 0 & $1639-k$ & $807-k$ \\
1&1&1&1&4&8& 0 & 0 & 3175 & 1575 \\
1&1&1&1&8&12& 0 & 0 & 7784 & 3879 \\
1&1&5&5&8&20& 0 & 6 & 2493 & 1237 \\
\hline
1&1&9&9&4&4&$0\leq k\leq 3$ & 0 & $1415-k$ & $695-k$ \\
5&5&13&13&4&4&$0\leq k\leq 5$ & 0 & $295-k$ & $135-k$ \\
1&1&13&13&4&8& $0\leq k\leq 2$ & 0 & $983-k$ & $1991-k$ \\
1&1&21&21&4&16& $0\leq k\leq 1$ & 0 & $3927-k$ & $1951-k$\\
5&5&25&25&4&16& $0\leq k\leq 2$ & 0 &$487-k$ & $231-k$\\
1&1&21&21&12&28& $0\leq k\leq 2$ & 0 &$2983-k$&$1479-k$\\
1&1&37&37&8&28& 0 & 0 & 5911 & 2943 \\
1&1&53&53&20&32&0 & 0 & 6055&3015 \\
21&21&49&49&4&24&$0\leq k\leq 9$& 0 &$263-k$&$119-k$ \\
1&1&69&69&16&52&0 & 0 & 10087 & 5031 \\
\hline
\end{tabular}
\bigskip
\caption{\label{tab:results}The weights of the ambient weighted projective spaces of the Calabi--Yau 4-folds and Betti numbers of the resulting $\mathrm{Spin}(7)$-manifolds.}
\end{table}

The sets of Betti numbers realised by manifolds constructed in this thesis are all distinct from those of compact manifolds with holonomy $\mathrm{Spin}(7)$ already known.
Also it should be noted that the example with $b^4=15118$ and $b^4_-=5031$ has the largest known value of $b^4$ or $b^4_-$ for a compact manifold with holonomy $\mathrm{Spin}(7)$.

For a fixed weighted projective space if we consider the family of Calabi--Yau 4-folds, which are fixed by the antiholomorphic involution $\tau$, it is interesting to note that the resulting $\mathrm{Spin}(7)$-manifolds coming from resolutions of different Calabi--Yau 4-folds in the family all have $b^4_-+b^2+1$ constant, which is the dimension of the Conformal Field Theory moduli space \cite{shatashvili1995superstrings}.

The next step for looking for Calabi--Yau orbifolds satisfying Condition \ref{cond:Y} would be to look for hypersurfaces in toric varieties coming from reflexive polytopes, which is invariant under an involution of the lattice. 
One could also look for orbifolds with more general types of singularities, which admit resolutions with holonomy $\mathrm{Spin}(7)$, for example the singularities in \cite[Sec. 4.3]{Joyce:1999fk}.
The methods we have described for calculating the Betti numbers would immediately apply. 

It should be noted that there is a notion of mirror symmetry for Calabi--Yau hypersurfaces in toric varieties determined by reflexive polytopes. 
The mirror polytope will also admit a non-standard antiholomorphic involution but there will be a choice involved. 
Also the mirror Calabi--Yau will, in general, not have the type of singularities we described in Section \ref{sec:ReviewofConstruction}.

\bibliographystyle{abbrv}
\bibliography{transfer}

\end{document}